\documentclass[10pt regno]{amsart}
\usepackage{amscd}
\usepackage{color}
\usepackage{amsfonts}
\usepackage[all]{xy}
\usepackage{amsmath,amssymb,amsthm,latexsym}
\usepackage{amscd}
\usepackage{bbold}
\usepackage{bm}
\newcounter{defcounter}
\usepackage{amsfonts}
\usepackage{amsmath}
\usepackage{array}
\usepackage{enumerate}
\usepackage{graphicx}
\usepackage{amssymb}
\usepackage{amsthm}
\usepackage{xcolor}

\usepackage{verbatim,amscd,amsmath,amsthm,amstext,amssymb,amsfonts,latexsym,lscape,rawfonts
}

\usepackage{verbatim,amscd,amsmath,amsthm,amstext,amssymb,amsfonts,latexsym,lscape,rawfonts
}

\newtheorem{theorem}{Theorem}[section]
\newtheorem{proposition}[theorem]{Proposition}
\newtheorem{definition}[theorem]{Definition}
\newtheorem{corollary}[theorem]{Corollary}
\newtheorem{lemma}[theorem]{Lemma}
\numberwithin{equation}{section}
\theoremstyle{remark}
\newtheorem{remark}[theorem]{Remark}
\newtheorem{example}[theorem]{\bf Example}
\newcommand{\R}{\mathbb{R}}

\begin{document}

\title[Resnikoff silver numbers]{Resnikoff silver numbers \\and tilings of the half-line\\
(Dedicated to the memory of H.L.~Resnikoff)}
\author{Josef F.~Dorfmeister}
\address{Fakult\"{a}t f\"{u}r Mathematik, TU M\"{u}nchen, Boltzmannstr. 3,
D-85747 Garching, Germany}
\email{	josef.dorfmeister@tum.de}
\author{Sebastian Walcher}
\address{Fachgruppe Mathematik, RWTH Aachen, D-52056 Aachen, Germany}
\email{walcher@mathga.rwth-aachen.de}
\maketitle
\begin{abstract} 
Building on work by H.L.Resnikoff we consider {\em (Resnikoff) silver numbers}, which generalize the familiar golden number. By definition, a silver number is the largest positive root of a certain polynomial called {\em silver polynomial}. In turn, a corresponding companion matrix of a silver polynomial gives rise to a well known construction of inflationary 
tilings of the  (non-negative) real half-line, via an iteration of inflation and substitution. Resnikoff noted for the golden number $\phi$ that this tiling corresponds to the set of what he called $\phi$-integers. We generalize this result for a special class of silver numbers, the {\em distinguished silver numbers}, by showing that the integers for a distinguished silver number give rise to a tiling, of which we provide a precise description. For the general problem, whether the integers for an arbitrary silver number give rise to a tiling, we cannot give a general answer, but we show that tilings are obtained if and only if the differences of silver integers satisfy a (rather weak looking) non-accumulation condition. If tilings of this type exist for certain (necessarily non-distinguished) silver numbers, they would seem to form a class of inflationary tilings that differs from those obtained by inflation and substitution.\\
In an Appendix we recall necessary notions and -- mostly known -- results, including the inflation-substitution construction principle for (one dimensional) inflationary tilings, in an elementary manner. For the readers' convenience we also collect the pertinent facts about non-negative matrices, thus the construction is accessible with only basic prerequisites from linear algebra and analysis. 
Finally, in our setting we give a detailed proof of a non-periodicity result that goes back to Penrose.

\end{abstract}

\tableofcontents

\section{Introduction}
In a series of preprints published around 2015, among them \cite{Res} and \cite{ResFT}, the late Howard L.\ Resnikoff discussed tilings of the line and the plane, with a view toward applications to Function Theory, in particular generalizations of special meromorphic functions and functional equations. Toward this goal, he also considered positional representations as well as inflationary tilings. Specializing to tilings of $\mathbb R$ and the half-line $[0,\,\infty)\subset \mathbb R$, he introduced the general notion of {\em silver numbers} as generalizations of the golden number. By definition a silver number $\rho$ satisfies a polynomial equation of the form $\rho^N=\sum_{i=1}^Nb_i\rho^{N-i}$, with all $b_i\in \{0,\,1\}$ $b_N=1$ and at least one further $b_i=1$.
Several special silver numbers, among them the  tribonacci number, the supergolden number and the plastic number, are known in the literature and are still being investigated by various authors\footnote{Facts and notions which appear here without further explanation, will be properly introduced in later sections.} Distinguished silver numbers are characterized by the polynomial equations $\rho^N=\sum_{i=1}^N\rho^{N-i}$, and are also known as $N$-generalized Fibonacci numbers.\\
Both \cite{Res} and \cite{ResFT} were intended as first steps, but, sadly, Resnikoff could not finish this work. In the present paper we take up one unfinished piece of work, and aim to bring it to a certain closure. While our manuscript is intended as a homage to Resnikoff, in our opinion the underlying ideas and techniques are of wider interest.\\
From here on we consider tilings of the real half-line, which may be characterized by the strictly increasing sequence of the tile endpoints. Our focus is on inflationary tilings. The construction of such tilings by a well defined iteration of inflation and substitution is well-known among specialists; see Resnikoff \cite{Res} for a concrete version, and Barge and Diamond \cite{BarDia} for an abstract version. But like all recursive constructions, such an approach cannot yield an explicit ``closed form''expression for the endpoints of the tiles. Taking a different approach, Resnikoff hinted at, and verified for the case of the golden number, an alternative construction of inflationary tilings, via ``integers'' (Resnikoff's terminology) related to the inflation multiplier. This approach by Resnikoff will be discussed in detail and extended in the present manuscript.\\
In Section \ref{sec:silverbase} we introduce the notions of silver polynomial, silver number; distinguished silver polynomial and distinguished silver number, and present some of their properties, either giving proofs or referring to the literature. In particular, every silver number $\rho$ satisfies $1<\rho<2$, and is an algebraic integer.
In Section \ref{sec:sigint}, for any real $\sigma\in (1,\,2]$ we adopt terminology of Resnikoff and call every number of the form $\sum_{i=0}^Nd_i\sigma^i$, with $n\in \mathbb N_0$, $d_0,\ldots,d_N\in \{0,\,1\}$ a $\sigma$-integer. In particular, when $\sigma$ is a silver number then we also speak of silver integers.
We proceed to discuss integers for distinguished silver numbers $\rho$, 
 derive a normal form representation for these, and obtain a complete characterization of their ordering, including expressions for the differences between consecutive $\rho$-integers. From this characterization, in turn, we find in
Section \ref{sec:inftile} that the $\rho$-integers for a distinguished silver number $\rho$ form the endpoints of an inflationary tiling of the half-line. On the way towards this result we present a simple classification of all silver polynomials for which the companion matrix is primitive. Moreover, we
characterize tilings of the half-line for every silver polynomial with primitive companion matrix, and show that they are inflationary, the multiplier being some power of the corresponding silver number. We proceed to discuss tilings for silver polynomials with non-primitive companion matrices
of silver polynomials and show that every such polynomial $P(X)$ can be naturally obtained from a silver polynomial $Q(X)$ with primitive companion matrix by replacing $X$ by setting $P(X) = Q(X^d)$, with some integer $d>1$. Since silver polynomials only have coefficients in $\{0,1\},$ the tilings induced by $Q(X)$ and $P(X)$ are in $1-1$ correspondence.
Section \ref{sec:nondis} is dedicated to silver numbers that are not distinguished, to their corresponding integers and tilings. We look more closely at the two non-distinguished silver numbers of degree three, the super-golden number and the plastic number. For non-distinguished silver numbers it is not clear whether the integers form a tiling with finitely many prototiles. If this were the case for some non-distinguished silver number, one would obtain a construction of inflationary tilings that cannot be obtained from a inflation-substitution iteration. We have to leave this existence problem open, but for silver numbers that  also have the Pisot property we state and prove a necessary and sufficient criterion. \\

Finally Section \ref{appsection}, the Appendix, collects a number of notions and results on tilings, on non-negative matrices, on the construction of inflationary tilings via non-negative integer matrices, and non-periodicity. 
We give for our special setting a detailed proof of a non-periodicity result by Penrose, namely that tilings of the non-negative real half-line constructed from general
primitive integer matrices with irrational spectral radius are never periodic.
Throughout the Appendix the presentation is elementary and intended to give a quick introduction to non-expert readers.

\section{Silver numbers and  distinguished silver numbers}\label{sec:silverbase}

\subsection{Basic definitions}

We first recall some notions and facts from Resnikoff  \cite{Res}.
\begin{definition}Let $N>1$ be an integer, and let $b_1,\ldots,b_N \in \{0,\,1\}$, with $b_N=1$ and  $\sum b_k > 1.$ The polynomial
\begin{equation}\label{silverpoly}
    P(x)=x^N-\sum_{j=1}^N b_jx^{N-j}
\end{equation} 
is called a {\em Resnikoff silver polynomial} of degree $N$, associated to $b:= (b_1,\ldots,b_N)$.
\end{definition}
\begin{example}\label{exsilpo} Some silver polynomials:
\begin{itemize}
\item The only silver polynomial of degree two is $x^2-x-1$.
\item There are three silver polynomials of degree three, viz. $x^3-x^2-x-1$, $x^3-x^2-1$ and $x^3-x-1$. One quickly verifies that they are all irreducible over the rationals $\mathbb Q$.
\item The silver polynomial $x^4-x^2-x-1=(x+1)(x^3-x^2-1)$ is reducible.
\end{itemize}
\end{example}
 The following is adapted from Resnikoff \cite{Res}, Lemma 1 and Theorem 6.
\begin{lemma}\label{reslem} Every silver polynomial has a root in the open interval $(1,2)$, and no real roots $>2$. Moreover, if a rational root exists then it is equal to $-1$. 
\end{lemma}

\begin{proof}With $P$ as in \eqref{silverpoly}, one sees $P(1)=1-\sum_{j=1}^N b_j<0$, and for all $z\geq 2$ one has
\[
P(z)\geq z^N-\sum_{j=1}^Nz^{N-j}=z^N-\dfrac{z^N-1}{z-1}\geq z^N-(z^N-1)>0. 
\]
This shows the first claim. As to the second, by a theorem of Gauss every reducible normalized polynomial with integer coefficients is reducible over the integers. (See Lang \cite{Lang}, IV\S2, Thm. 2.1 and Cor 2.2.) In particular the constant coefficients of both factors must be $1$ or $-1$.
\end{proof}

\begin{definition} ({Silver numbers}) \label{silvernumber}
\begin{itemize}
    \item We call $\sigma$ a {\em Resnikoff silver number} if it is the largest positive root of some silver polynomial. 
    \item We call a Resnikoff silver number, $\rho$, a {\em distinguished Resnikoff silver number}, if it is the largest positive root of the distinguished silver polynomial with $b_1 = \dots = b_N = 1,$ i.e.
\begin{equation} \label{dist-silverpoly}
\rho^N = \rho^{N-1}+\dots + \rho + 1 = \sum^{N}_{i=1} \rho^{N-i}.
\end{equation}
\end{itemize}
\end{definition}
The particular name ``silver number'' was coined by Resnikoff, and we will adopt it in the present article. Other names (for certain cases) are in use; see below.
\begin{example}\label{colors}
Some silver numbers, and some other names:
\begin{itemize}
    \item  The definition includes, when $N=2$, the golden number $\phi$.
    \item In the literature one encounters the notion of ``silver ratio'' for $\zeta=1+\sqrt 2$, the positive root of the polynomial $x^2-2x-1$. This needs to be distinguished from Resnikoff's (and our) terminology.
    \item Distinguished silver numbers are generally known by the name of $N$-generalized Fibonacci numbers; see e.g. Dresden and Du \cite{Dresden}. Those of degree three are also known as tribonacci numbers (see e.g. Feinberg \cite{Feinberg}), and for higher degrees the names tetranacci, pentanacci and so on have been proposed.
    \item There are two non-distinguished silver numbers of degree three, viz.\ the {\em supergolden number} $\psi$, i.e.\ the positive root of the polynomial $x^3-x^2-1$ (see e.g. Lin \cite{Lin}), and the {\em plastic number} $\rho$, which is the positive root of $x^3-x-1$ (see e.g.\ van der Laan \cite{Laan} and Shannon et al. \cite{SAH}). Both of these have been investigated from various aspects; we will discuss them further later on.
\end{itemize}
   
\end{example}

\subsection{Distinguished silver numbers}
 In this and the following two sections, unless the opposite is stated explicitly, we will always consider distinguished Resnikoff silver numbers.
In this subsection we collect, in part with proofs, known facts about distinguished silver numbers, including a refinement of Lemma 1. 
\begin{theorem} \label{distpol}
Let  
\[
P_N(X) = X^N - X^{N-1}- \dots - X - 1 = X^N - \sum^{N}_{i=1} X^{N-i}
\]
be the distinguished silver polynomial of degree $N \geq 2$.

\begin{enumerate}
\item The polynomial $P_N(X)$ has exactly one positive real root, which lies in the interval $(1,\,2)$, and will be called $\rho_N$. Moreover $\rho_N$ is irrational.
\item All other roots of $P_N(X)$ are contained in the open unit disk.
\item $P_N(X)$ is irreducible over the rationals $\mathbb Q$, and the roots of $P_N(X)$ are all different.
\end{enumerate}
\end{theorem}
\begin{proof}
    A part of the first statement is just a repetition of Lemma \ref{reslem}; the rest as well as the second statement is proven in Miller \cite{Miller}. As for the third, it suffices to show irreducibility, since an irreducible polynomial over a field of characteristic zero cannot have multiple roots: Its (formal) derivative is not zero, and existence of a multiple root would imply that the polynomial and its derivative have non-constant gcd; a contradiction to irreducibility. Thus assume that $P_N=Q\cdot R$ is reducible. Then by Gauss' theorem (Lang \cite{Lang}, IV\S2, Thm. 2.1 and Cor 2.2.), it is reducible over the integers; thus we may assume $Q,\,R\in\mathbb Z[X]$. By part (2), one of these polynomials, say $Q$ has only roots of modulus $<1$. But then the constant coefficient of $Q$ has modulus $<1$. We arrive at a contradiction, since this constant coefficient is a nonzero integer.
\end{proof}
\begin{remark}
\begin{enumerate}
\item    For further facts about distinguished Resnikoff silver numbers see Dresden and Du \cite{Dresden}, Miles \cite{Miles}, Miller \cite{Miller}, and Wolfram \cite{Wolfram}, among other sources.
Particularly interesting is an explicit  formula for $\rho_N$ given in \cite{Wolfram}. See also the recent preprint \cite{Mane} by Mane.
\item Part (2) of the Theorem means that $\rho_N$ is a Pisot number. By definition, a Pisot number is a real algebraic integer $>1$ such that all the other roots of its minimum polynomial have modulus $<1$ (see e.g. \cite{PiSal}).
\item The companion matrices of all distinguished silver numbers are primitive (see Theorem \ref{primmat}).
\end{enumerate}
\end{remark}

\begin{theorem}
The sequence $\rho_N, N = 2,3,....,$ of distinguished silver numbers is monotonically increasing and converges to $2$.
Consecutive sequence members satisfy
\begin{equation}
2 - \frac{1}{2^{N-1}} < \rho_N < 2 - \frac{1}{2^{N}}  < \rho_{N+1} < 2 - \frac{1}{2^{N+1}}.
\end{equation}
In particular we have:

\begin{equation}
| \rho_{N+1} - \rho_N| < 
\frac{1}{2^{N-1}}  - \frac{1}{2^{N+1}} = \frac{1}{2^N}.
\end{equation}
\end{theorem}

\begin{proof}
 We first prove the inequality: $2(1 - \frac{1}{2^{k}})  < \rho_N \leq 2.$
The following proof follows a suggestion of D.A.Wolfram \cite{WolframPers}. The defining equation for $\rho = \rho_N$ is clearly equivalent to
$\rho =  1 + 1/\rho + .... 1/\rho^{k-1}. $ We thus need to show that the two functions $f(y) = y$ and
$g(y) = 1 + 1/y + .... 1/y^{k-1} $ intersect in the open interval $]1,2[$  exactly once. 
For this we note that $g(y)$ is strictly decreasing on $]0,\infty]$ and thus attains the minimum in $[1,2]$ at $y = 2$ and 
this minimum can easily be computed, yielding $2(1 - \frac{1}{2^{k}})$. Moreover, the maximum is attained at $y =1$
and takes the value $N$.
But, since $f(y)$ is strictly increasing in $[1,2]$, attaining the minimum $1$ at $y=1$ and the maximum $2$ at $y = 2$, 
there is exactly one intersection of the curves $f$ and $g$ on $]1,2[$ satisfying the inequalities as stated.

Next we prove the inequality: $\rho_N < 2 - \frac{1}{2^{k}}$: To this end we rewrite the defining equation for $\rho_N$ 
(equivalently) in the form $X^{N+1} - 2 X^N + 1 = 0$. Hence $2-\rho_N = \frac{1}{\rho_N^{N}} \geq \frac{1}{2^{N}}$
and the claim follows.

The last inequality follows directly from the first one.
\end{proof}

As an immediate consequence of the above we obtain a result first shown by Dresden (\cite{Dresden}, Lemma 1).
\begin{corollary} 
Let $\rho_N$ be  a distinguished silver number. Then
\begin{equation}
2 - \frac{1}{N}  \leq \rho_N \leq 2 , \hspace{2mm} \mbox{if } \hspace{2mm} N = 2,3,
\end{equation}
\begin{equation}
2 - \frac{1}{3N}  \leq \rho_N \leq 2 , \hspace{2mm} \mbox{if } \hspace{2mm} N  \leq 4.
\end{equation}
\end{corollary}







\section{$\sigma-$integers}\label{sec:sigint}

\subsection{The definition}

From  Resnikoff \cite{Res} we adopt the following terminology:

\begin{definition}\label{defsigint}  Let $\sigma\in \mathbb R$, $1<\sigma\leq 2$. A non-negative real number $x$ is said to be a {\em $\sigma-$integer} if it  has the form
\begin{equation}
x = c_0 \sigma^n + c_1 \sigma^{n-1} + \dots + c_{n-1} \sigma + c_n   = \sum^{n}_{i=0} c_i \sigma^{n-i},
\end{equation}
with all $c_i \in \{0,1\}$.\\
If $x\not=0$ and $c_k =0$ for all $k< r$ and $c_{r} = 1$, for some $0\leq r\leq n$ then we will say that the representation of $x$ has degree $n-r$. When ($x=0$ and) all $c_k=0$, we assign the degree $-\infty$ wherever convenient.
\end{definition}

Notably, the $\sigma$-integers exhibit a kind of inflation property, which we record:
\begin{remark} \label{multsigma}
\begin{enumerate} 
\item If $x$ is a $\sigma-$integer, then $\sigma x$ also is a $\sigma-$integer.
\item By the item just above, for a silver number $\rho$ the set of $\rho-$integers
exhibits at least one feature the endpoints of some inflationary tiling with multiplier $\rho$. This raises the question whether the corresponding $\rho-$integers form a tiling.
\item Another open question is, whether the $\rho-$integers of a silver number all are endpoints of an inflation-substitution tiling obtained from the companion matrix of the corresponding silver polynomial in section \ref{appsection}.
\end{enumerate}
In the next section and in subsection \ref{goldennumber} we will discuss all which seems to be known with regard to the questions raised above.
\end{remark}

\begin{remark}
Resnikoff \cite{Res,ResFT} provides two motivations for considering $\sigma$-integers when $\sigma$ is a silver number:

a) The first  motivation lies in their relevance for tilings of the half-line $[0,\infty)$  (which will be the main topic of the present paper).

b) The second motivation lies in the relevance for the construction of certain meromorphic functions (generalized zeta functions) in the complex plane.
Actually, Resnikoff's interest was primarily in $2-$D tessellations.
So this paper is the $1-$D version of what Resnikoff ultimately may have planned to investigate.
\end{remark}
\vspace{3mm}

Resnikoff's particular naming ``$\sigma$-integers'' is motivated by the case $\sigma=2$: The set of $2-$integers is the set of all non-negative integers.
\begin{remark}
\begin{enumerate}
\item For $1<\sigma\leq 2$ Resnikoff also notes a correspondence to positional representations with respect to the basis $\sigma$: Every real number $z<1$ admits a ``$\sigma$-adic'' representation
\[
z=\sum^{\infty}_{i\geq 1} c_i \sigma^{-i}.
\]
\item We note that for silver numbers (distinguished or not), as well as for other types of numbers, the above representation is not unique.
\item For the corresponding dynamical systems and their properties see, among others, R{\'e}nyi \cite{Ren}, Gel'fond \cite{Gel}, Parry \cite{Parry}, Cigler \cite{Cigler}. 
\item Likewise, for a silver number $\sigma$ (distinguished or not) the representation of a $\sigma-$integer as given in Definition \ref{defsigint} is obviously not unique, since suitable multiples of the defining silver polynomial for $\sigma$ may be added.
\end{enumerate}
\end{remark}
\subsection{Distinguished silver numbers: Normal form for integers}
With the golden number 
$\phi = \rho_2$, Resnikoff defined 
for each $\phi-$integer a representation in ``normal form" and showed its existence and uniqueness.
Here we generalize this result for representations with respect to arbitrary distinguished silver numbers.
\begin{definition}  Let $\rho = \rho_N$ be a distinguished silver number.
A representation  $  x = \rho^n + \sum^{n}_{i=1} c_i \rho^{n-i} $ is said to be in {\em normal form } if
\[
\Pi^{j + N-1}_{k=j}c_k =0\text{   for all  }j,
\]
where we set $c_0 = 1$ and $ c_k = 0$ for $k>n$.\\
(In other words, every string $c_j,\ldots,c_{j+N-1}$ of $N$ consecutive coefficients contains a zero.)
\end{definition}

\subsection{Existence of normal forms}
The definition of a distinguished silver number immediately implies
\begin{equation} \label{gen-dist-silverpoly}
\rho^{l+1} = \rho^{l}+ \dots + \rho^{l-N+1} = \sum^{N-1}_{i=0} \rho^{l-i}
\end{equation}
for all $l>N$.

For later use we state a simple consequence of (\ref{dist-silverpoly}):
\begin{lemma} \label{variations}
Let $\rho_N = \rho$ be a distinguished silver number. Then for $1 \leq m \leq N-1$ we have
$$ \rho^m -\rho^{m-1} - \dots - \rho - 1 =
\frac{1}{\rho} +  \frac{1}{\rho^2}+ \dots +  \frac{1}{\rho^{N-m}}.$$  
\end{lemma}
\begin{proof}
This follows directly from  (\ref{dist-silverpoly}) by successively dividing by $\rho$ and splitting the 
total sum into a part with non-negative powers of $\rho$ and a part with negative powers of $\rho$.
\end{proof}
We also state the obvious

\begin{lemma}
Let $\rho_N = \rho$ be a distinguished silver number. Let $x\not=0$ be a $\rho-$integer with a representation of degree $n$. If $n< N$, then this representation is in normal form.
\end{lemma}

It is not a priori clear 
that for each distinguished silver number $\rho_N = \rho$ each  $\rho-$integer has a normal form.
A positive (and somewhat stronger) answer is given in

\begin{theorem}
Let $\rho_N = \rho$ be a distinguished silver number. Let $x\not=0$ be a $\rho-$integer with a given representation of degree $n\geq 0$. Then there exists a 
representation of $x$ in normal form of degree $n$ or $n+1$.
\end{theorem}

\begin{proof}
The proof will be given by induction on  $n$.

The cases $n \leq  N-1$ are trivial by the lemma above.

Assume now that $n\geq N$ and the claim has been proven for representations of degree $n -1$, and let the  $\rho-$integer $x$ 
have a representation of degree $n$. Then we have a representation of the form
$$x = c_0 \rho^n + c_1 \rho^{n-1} + \dots + c_{n-1} \rho + c_n , \hspace{2mm} \mbox{with} \hspace{2mm} c_0 = 1, 
\hspace{2mm} \mbox{and } \hspace{2mm}  c_j \in \{0,1\},  1 \leq j \leq n.$$

\noindent Let $k$ be the smallest integer so that $c_k = \cdots = c_{k + N-1}= 1$.

\noindent If no such $k$ exists, then we already have  $x$ in normal form.

If $k = 0$, then $c_0 =  \cdots = c_{N-1} = 1$.
Now (\ref{gen-dist-silverpoly}) yields for $l = n$ that $\rho^{n+1} = \rho^{n}+ \dots + \rho^{n-N+1}$ holds.
As a consequence  $x = \rho^{n+1} + y$, with $y=c_N \rho^{n-N} + \dots$ given by a representation of degree $\leq n-N$. By induction hypothesis we can replace the representation of $y$ by its normal form 
 $y =\rho^m + \dots + \tilde{c}_m$  with  $ m \leq n-N+1 \leq n-1$.
Therefore $x = \rho^{n+1} + y = \rho^{n+1} +  \rho^m + \dots + \tilde{c}_m$  is a normal form of $x$.

Assume now $0 < k$.  Then $c_{k-1} = 0$ and 
$$x = \{ \rho^n + \dots + c_{k-2} \rho^{n-k+2} \} + 0 + 
\{\rho^{n-k} + \dots +\rho^{n-k - N+1} \} + \sum_{i\geq k+N}c_i\rho^{n-i}.$$

Using (\ref{gen-dist-silverpoly})
we derive $ \rho^{n-k} + \dots + \rho^{n-(k+N-1)} = \rho^{n-k+1}.$ This gives a representation
$$x =\{ \rho^n + \dots + c_{k-2} \rho^{n-k+2} \} +  \rho^{n-k+1} + \{ c_{k+N}  \rho^{n-k-N} + ...\}.$$

The last term has highest degree $\leq n-k-N < n-1.$ We can therefore apply the induction hypothesis
and obtain a normal form of degree $\leq n-k-N+1 < n-k -1.$ As a consequence, there is a gap of at 
least two steps between the first two terms and the normal form of the last term. 
Factoring out $\rho^{n-k+1}$ from $y = \{ \rho^n + \dots + c_{k-2} \rho^{n-k+2} \} +  \rho^{n-k+1} $ we obtain
$y = (\rho^{k-1} + \dots + 1)  \rho^{n-k+1} .$ Since $k<n$ we also have $k-1 < n-1$ and we can apply
the induction hypothesis to the first factor of $y$. We thus obtain this first factor of $y$  in 
normal form of degree $ k -1$ or $k$. This yields a normal form of $y$ of degree $n$ or $ n+1$ and altogether 
a normal form of $x$ of degree $n$ or $n+1$.
\end{proof}

We finish this subsection with a result which will be of crucial importance in the next subsection.

\begin{theorem} \label{crucialinequ}
Let $\rho_N = \rho$ be a distinguished silver number. Let $m,\,n$ be non-negative integers with $n>m$ and  $y = \rho^m + \sum_{i=1}^m d_i \rho^{m-i}$  any distinguished $\rho-$integer in normal form.
Then 
$ \rho^{n} > \rho^m +   \sum_{i=1}^m d_i \rho^{m-i}.$
\end{theorem}
\begin{proof}
It suffices to prove this for the case $n=m+1$, since $\rho>1$. For $m=0$ the claim is trivial and for $1 \leq m \leq N-2$ the claim follows from Lemma \ref{variations}, since
$\rho^{m+1} - (\rho^m + \sum_{i=1}^m d_i \rho^{m-i}) \geq \rho^{m+1} - (\rho^m +   \sum_{i=1}^m  \rho^{m-i} ) > 0.$
For $m = N-1$ we use $\rho^{N} - (\rho^{N-1} + \sum_{i=1}^{N-1} d_i \rho^{N-1-i}) >
 \rho^{N} - (\rho^{N-1} +   \sum_{i=1}^{N-1} \rho^{N-1-i} ) = 0$ 
  Note that here  the first inequality is a
 consequence of the fact that $y$ is in normal form for $m \leq N-1$ and that therefore one of the $d_i$ vanishes.
 The last equality is the definition of $\rho.$
 
 Now we prove the claim by induction on $m = N-1, N,...$.
 Note that the base of the induction was already proved above.
 We thus assume  $y = \rho^m + \sum_{i=1}^m d_i \rho^{m-i}$ is a $\rho-$integer in 
 normal form with $m \geq N$ and that the claim has already been proven for all $\rho-$integers in normal form
  with a representation of degree $\leq m-1$.
  Since $y$ is in normal form, there exists a smallest integer $k\geq 1$ such that $d_k =0$ holds. 
  For the same reason, $k \leq N-1.$
  
  Next we split the representation of $y$ into two parts: 
\[
y =  \{ \rho^m + \sum_{i=1}^{k-1} d_i \rho^{m-i}\} +
   \{ \sum_{i=k}^{m} d_i \rho^{m-i}\}.
\]
 Considering the second term we know $d_k = 0$ by the definition of $k$. Hence 
   $ \sum_{i=k}^{m} d_i \rho^{m-i} =  \sum_{i=k+r}^{m} d_i \rho^{m-i} $, where $r \geq 1$ such that
    $d_k = \dots d_{k+r-1} = 0$ and $d_{k+r} = 1.$ Then 
\[
\begin{array}{rcl}
      \sum_{i=k}^{m} d_i \rho^{m-i} &=&  \sum_{i=k+r}^{m} d_i \rho^{m-i} =
      \rho^{m-k-r} + d_{k+r+1} \rho^{m-k-r-1} + \dots + d_m  \\
&=& 
      \rho^{m-k-r} + \sum_{j=1}^{m-k-r} d_{k+r + j} \rho^{m-k-r-j} < \rho^{m-k-r+1} \leq \rho^{m-k}.
\end{array}
\]
      
      Here we have applied  the induction hypothesis in the next to last step  of this argument 
      for $M = m-k-r$  and have used a trivial estimate in the last step.
      
      Returning to the expression for $y$ as a sum of two terms and using the estimate for the second term
      proven just above we infer $y < \{ \rho^m + \sum_{i=1}^{k-1} d_i \rho^{m-i}\} + \rho^{m-k} \leq  
       \sum_{j = 0}^{k}  \rho^{m-j} \leq  \sum_{j = 0}^{N-1}  \rho^{m-j} = \rho^{m+1},$ proving the assertion.
\end{proof}

\subsection{Orderings of $\rho-$integers}

In this subsection we will describe what the ordering of $\rho-$integers in $\R_{\geq 0}$ means 
in terms of their coefficients.

\begin{theorem} \label{ordering}
Let $\rho_N = \rho$ be a distinguished silver number. Let $x = \rho^n + \sum_{i=1}^n c_i \rho^{n-i}$ and $y = \rho^m + \sum_{i=1}^m d_i \rho^{m-i}$  
be $\rho-$integers in normal form. Then
\begin{itemize}
\item $x > y$, if $n > m.$
\item $x > y$, if $n=m$, and for $k = \min \{i; c_i \neq d_i \}$ we have  $c_k = 1$ and $d_k = 0.$
\end{itemize}
\end{theorem}

\begin{proof}The first statement is a direct consequence of Theorem \ref{crucialinequ}, in view of $x\geq \rho^n$.
The second statement follows from the first: Since $x>y$, not all coefficients $c_i$ of $x$ coincide 
with the corresponding coefficients $d_i$ of $y$. So write $ x = h_k + \tilde{x}$ and $ y = h_k + \tilde{y},$ where
$h_k =  \rho^n + \sum_{i=1}^{k-1} c_i \rho^{n-i} =  \rho^n + \sum_{i=1}^{k-1} d_i \rho^{n-i}$  with
$c_i = d_i$ for $i =1 , \dots, k-1,$ but $c_{k} \neq d_{k}$. Then $x > y$ is equivalent to $\tilde{x} > \tilde{y}$.
Moreover, degree $\tilde{x} > $ degree $\tilde{y}$. Now the first statement is applicable.
\end{proof}

\begin{corollary}\label{unique}
If $\rho_N = \rho$ is a distinguished silver number, then each $\rho-$integer has exactly one representation in normal form. 
\end{corollary}

\begin{proof} By Theorem \ref{ordering}, different normal form representations imply inequality.
\end{proof}
\begin{proposition}\label{largest}
Let $\rho_N = \rho$ be a distinguished silver number. Let $n$ be a non-negative integer and $x$ the largest $\rho$-integer in normal form that admits a representation of degree $n$. Then, with\footnote{Here for a real number $x$ the expression $\lfloor x\rfloor$ denotes the Gauss bracket of $x$, i.e. the greatest integer not exceeding $x$.} 
\[
k^*:=\left\lfloor\dfrac{n+1}{N}\right\rfloor,
\]
one has
\begin{equation}\label{largestrep}
x=\sum_{i=1}^{k^*}\rho^{n+1-iN}\left(\rho^{N-1}+\cdots +\rho\right)+R,
\end{equation}
with remainder
\begin{equation}\label{remainder}
R=\sum_{\ell=0}^{n-k^*N}\rho^\ell.
\end{equation}
Explicitly
\begin{equation}\label{remexplicit}
R=\sum_{\ell=0}^{r-1}\rho^\ell,
\end{equation}
with $r$ the remainder upon dividing $n+1$ by $N$ in integers. 
\end{proposition}
\begin{proof}
For $n<N-1$ the largest $\rho$-integer (in normal form) of degree $n$ is just $\sum_{\ell=0}^n\rho^\ell$.
Note that this coincides with (\ref{largestrep}), since $k^* = 0$. \\
For $n=N-1$ every $\rho$-integer in normal form of degree $N-1$ has the form $\sum_{i=0}^{N-1} d_i\rho^{N-1-i}$, with at least one $d_i=0$. Clearly the largest of these is $\sum_{i=0}^{N-2}\rho^{N-1-i}$, as asserted 
(note $k^*=1$, hence the remainder sum vanishes).\\
Now proceed by induction. We have $n\geq N$, hence
\[
x=\rho^n+\sum_{i=1}^nc_i\rho^{n-i}=\rho^n+\sum_{i=1}^{N-1}c_i\rho^{n-i}+\widetilde x,
\]
where
\[
\widetilde x:= \sum_{i=N}^nc_i\rho^{n-i}.
\]
Since $x$ is in normal form and $c_0 = 1$, we have  $c_1\cdots c_{N-1}=0$.
Since  $\rho^n+\sum_{i=1}^{N-1}c_i\rho^{n-i} $ and $\rho^n+\sum_{i=1}^{N-2}\rho^{n-i}$ 
 are in normal form, we can apply  Theorem \ref{ordering}  and obtain for all possible choices of 
 the $c_j$ the inequality
\[
\rho^n+\sum_{i=1}^{N-1}c_i\rho^{n-i}\leq \rho^n+\sum_{i=1}^{N-2}\rho^{n-i}.
\]
As a consequence,  equality holds by the maximality of $x$. Thus
\[
x=\rho^{n+1-N}(\rho^{N-1}+\cdots +\rho) +\widetilde x.
\]
Since $x$ is in normal form, so is $\widetilde x$. To ensure applicability of  the induction hypothesis to $\widetilde x$, we show that $\widetilde x$ is maximal among all $\rho$-integers that have the same degree as $\widetilde x$, which we denote by $n-N-k$ with some $k\geq 0$. Assume that there exists some $\rho$-integer $\widetilde y$, of the same degree, satisfying $\widetilde x<\widetilde y$. Then
\[
y:=\rho^{n+1-N}\left(\rho^{N-1}+\cdots+\rho\right)+\widetilde y=\rho^{n+1-N}\left(\rho^{N-1}+\cdots+\rho\right)+0\cdot\rho^{n+1-N}+\widetilde y
\]
is in normal form and satisfies $x<y$; a contradiction to the maximality of $x$. Now the induction hypothesis shows the assertion.
\\\ To make sure that here the remainder $R$ has the asserted range of summation
we take a closer look at $k^*$ and $R$: Division with remainder in integers yields
\[
n+1=d\,N+r,\quad 0\leq r\leq N-1.
\]
From this,
\[
\dfrac{n+1}{N}=d+\dfrac{r}{N};\quad \left\lfloor\dfrac{n+1}{N}\right\rfloor=d; \quad N\cdot \left\lfloor\dfrac{n+1}{N}\right\rfloor=dN=n+1-r;
\]
and finally 
\[
n-k^*N=r-1.
\]
The claim follows.
\end{proof}

The following is the principal result about $\rho$-integers.
\begin{theorem}\label{rhotiles}
Let $\rho_N = \rho$ be a distinguished silver number. Let $x,\,y$ be {\em consecutive} $\rho$-integers in normal form, $x>y$.
\begin{itemize}
\item  If $y=0$, then $x = 1.$\\
\item  If $ x=z+\widetilde x,\quad y=z$, where $z$ is a $\rho-$integer, then $\widetilde x = 1.$\\
\item  If $ x=z+\widetilde x,\quad y=z + \widetilde y$, where $z$ is a $\rho-$integer and $\widetilde{y} \neq 0$, then
$\widetilde x$ is of higher degree than $ \widetilde y$.\\
\end{itemize}
In case $\widetilde y \neq 0$ one has
\begin{equation}
x-y=\sum_{i=r+1}^N\rho^{r-i}=\dfrac{1}{\rho}+\cdots+\dfrac{1}{\rho^{N-r}},
\end{equation}
with $r$ the remainder upon division of $n+1$ by $N$.
\end{theorem}
\begin{proof} The case $\widetilde y=0$ is obvious, since $1$ is the smallest nonzero $\rho$-integer. 
\newline
For the rest of the proof we assume that $\widetilde x$ has degree $m$ and $\widetilde y$ has degree $n$ with $m>n.$
We thus have  $\widetilde x=\rho^{m}+\cdots.$ Now first, since $\widetilde x$ is the smallest $\rho$-integer greater than $\widetilde y$ we have $m=n+1$ and $\widetilde x=\rho^{n+1}$ by Theorem \ref{ordering}. Second, $\widetilde y$ is then the largest $\rho$-integer of degree $n$, which was determined in Proposition \ref{largest}. \\
For $n<N-1$ we get that $\widetilde y$ is the remainder, and 
\[
x-y=\widetilde x-\widetilde y=\rho^{r}-  \sum_{\ell=0}^{r-1} \rho^\ell=\dfrac{1}{\rho}+\cdots+\dfrac{1}{\rho^{N-r}}
\]
by Lemma \ref{variations}.\\
Now assume $n\geq N-1$. Then
\[
\begin{array}{rcccl}
x-y&=&\rho^{n+1}-\left(\rho^n+\cdots+\rho^{n-N+2}\right) &+& y^*\\
 &=& \rho^{n-N+1} &+& y^*,
\end{array}
\]
where $y^*$ has degree $n-N$ (see Theorem \ref{ordering}), and an obvious variant of Lemma \ref{variations} implies the second equality.\\
Now one can proceed by descent and reduce this to the case $n<N-1$, to finish the proof.
\end{proof}

The following statement is now obvious.
\begin{corollary}\label{distnoclus} Let $\rho_N = \rho$ be a distinguished silver number. 
Then 0 is not a cluster point of the set of differences of pairs of $\rho-$integers.
\end{corollary}

\subsection{Natural tilings from silver integers}
We note another immediate consequence of Theorem \ref{rhotiles}. The second statement in the proposition below holds because multiplication by 
a silver number $\sigma$  sends $\sigma$-integers to $\sigma$-integers
(see Remark \ref{multsigma}).

\begin{proposition}\label{inttile}
Let $\rho_N = \rho$ be a distinguished silver number. Then the $\rho$-integers are the endpoints of a tiling of the half-line $[0,\,\infty)$ by prototiles of lengths
\[
1,\quad \sum _{i=1}^{N-k}\rho^{-i},\quad 1\leq k\leq N-1.
\]
Moreover, this tiling is inflationary with multiplier $\rho$.
\end{proposition}
In other notation, the tiling presented in the proposition has 
prototiles of lengths, in descending order, 
\begin{equation}
1, \quad \dfrac{1}{\rho}+\cdots+\dfrac{1}{\rho^{N-1}}, \quad \dfrac{1}{\rho}+\cdots+\dfrac{1}{\rho^{N-2}}, \quad \ldots \quad ,\dfrac{1}{\rho}. 
\end{equation}

\section{Inflationary tilings associated to distinguished silver numbers.}\label{sec:inftile}

We want to compare
the tilings introduced in Proposition \ref{inttile} to tilings obtained by the construction via inflation and substitution.
For general facts about tilings, and about the inflation-substitution construction, we refer to Section \ref{appsection}.

\subsection{About primitivity of the partition matrix of a silver polynomial} \label{partitionmatrixofsilver}

We consider an arbitrary  
Resnikoff silver number $\sigma$ of degree $N$, 
\[
\sigma^N=\sum_{j=1}^N b_j\sigma^{N-j}.
\]

By definition, the corresponding Resnikoff silver polynomial $P_\sigma$ of degree $N$ satisfies $b_N = 1$, 
but also $b_s = 1$ for some integer $ 1 \leq s < N$.

Next we consider the non-negative companion matrix of  $P_\sigma,$  as spelled out explicitly in the form 
\begin{equation} \label{Uforsilver}
U=U_\sigma=\begin{pmatrix} b_1&b_2&\cdots&\cdots&b_N\\ 
            1 & 0& & &0\\
        0&\ddots&  & & 0\\
        \vdots & & & &\vdots\\
    0&\cdots & & 1&0
\end{pmatrix}.
\end{equation}
We follow Resnikoff \cite{Res}, equation (15), with this choice.
(Resnikoff calls this the Frobenius companion matrix.)
For a discussion of other natural choices see subsection \ref{subsec:compmat}.

Using the results about primitive matrices presented in  Theorem \ref{primmat} we can prove the following characterization of primitive as well as non-primitive companion matrices of silver polynomials.

\begin{theorem} \label{classifyprimitive}
\begin{enumerate}
    \item If $U_\sigma$ is the companion matrix of a Resnikoff silver polynomial, then  $U_\sigma$ is primitive if and only if
the labels $s_1, s_2, \dots N$ of all  non-vanishing $b_k$ that occur in the defining polynomial of  $\sigma$ have no non-trivial common divisor.
\item Conversely, $U_\sigma$ is not primitive if and only if there exists an integer $d>1$ and a silver polynomial $Q$ with primitive companion matrix, such that $P_\sigma(X)= Q(X^d)$. In particular, in addition to the silver number $\sigma = \sigma_P$, also $\sigma^d = \sigma_Q$ is a silver number, and $P_\sigma$ has more than one root of modulus $>1$.

\end{enumerate}
\end{theorem}
\begin{proof}
Item 1: Since we know that $U_\sigma$ is a non-negative matrix, by the  defining property of a primitive matrix (see Theorem \ref{primmat}) we need to show that some positive power of $U_\sigma$  has only positive entries.
Let $\mathcal{N} = \{ 1 \leq s_1 < \dots s_r = N \}$ denote the set of labels, for which the corresponding coefficient in the representation of $\sigma$ does not vanish. Note that  $\mathcal{N} $ has at least two entries.

Now there is a formula for the entries of  $U_\sigma^n$  stated in Theorem 3.1 of \cite{ChenLouck}.
A specialization of this formula to our case leads to show that for a fixed, but sufficiently high power $n$ of $U_\sigma$, 
one needs to find for any matrix position $(i,j)$  in $U^n_\sigma$ non-negative integers  $k_s, s \in  \mathcal{N},$ 
such that 

\begin{equation} \label{solution}
s_1k_{s_1} + \dots  + s_rk_{s_r} = n - i + j
\end{equation}
holds.

Recall that we have $s_{r} = N$ and note that we only can use positive $k_t$ in positions where $b_t \neq 0$, since otherwise the last factor in loc.cit. annihilates the whole term.

Also observe, if we have any (non-trivial) solution to this equation, then Theorem  3.1 of \cite{ChenLouck} implies that the related coefficient in the $(i,j)-$position  of  $U^n_\sigma$ is positive.
So we finally need to determine, when equation $(\ref{solution})$ has a solution. 

Case 1: Assume that  $s_1, \dots , s_N = N$  do not have any common divisor.
Then there exist integers  $h_{s_1}, \dots , h_{s_r}$ such that 
$$s_1 h_{s_1} +  \dots  + s_r h_{s_r} = 1.$$

Now consider positive integer coefficients $A_l, l \in \mathcal{N},$ and put  
$$ n = A_1 \cdot s_1 + \dots + A_r\cdot s_r .$$

Since we also want to have a representation of all integers $ n \pm j, j = 1, \dots , N-1,$ in this form
we need to choose the $A_j$ large enough.

This can be obtained as follows:  We clearly  have
 $$ n  \pm{ j} =(A_1 \pm{j} h_{s_1})  \cdot s_1 + \dots + (A_r \pm{ j} h_{s_r})  \cdot s_r .$$
 
 Since we need all coefficients occurring to be positive integers for all values of $j = 1, \dots , N-1.$
 We choose the $A_{s_r}$  such that  $0 < A_{s_r} \pm{ j} h_{s_r} $ for all $s_r \in \mathcal{N}.$

Then the family of coefficients  $k_{s_l} = (n - i + j) h_{s_l}$ solves the required equation,
thus finishing case 1.

Case 2: Assume now the coefficients $s_1, \dots , s_r = N$  have a non-trivial common divisor $\delta$.
Then all $n-i+j$ need to be divisible by $\delta$. But this is not possible. The first claim is proven.

Item 2: To prove the second claim, note that the characterizing condition for non-primitivity may be restated as
\[
d:={\rm gcd}\,\left\{j; \,b_j\not=0\right\}>1.
\]
In particular $d$ divides $N$; we have $N=d\cdot N^*$, hence we can write 
\[
P_\sigma=X^N-\sum_{\ell=^1}^{N^*}b_{d\ell}X^{N-d\ell}={(X^d)}^{N^*}-\sum_{\ell=^1}^{N^*}b_{d\ell}{(X^d)}^{N^*-\ell}=:Q(X^d),
\]
and the companion matrix of $Q$ is primitive by the definition of $d$.
\end{proof}
\begin{corollary} \label{simple}
Let $\sigma $ be a silver number with silver polynomial $P_N(X) = X^N - b_1 X^{N-1} + \dots b_{N-1} X + b_N.$
Assume that $b_1 = 1,$ or that $N$ is prime, or that $b_{s_1 }=b_{s_2}=1$ for two relatively prime  coeffcients $s_j$ satisfying  $1 < s_j < N.$
Then $U_\sigma$ is primitive.
In particular, if $\sigma$ is a distinguished silver number, then $U_\sigma$ is primitive.
\end{corollary}

 We now specialize the inflation-substitution construction of tilings (see subsection \ref{subsec:nonneg}, in particular Corollary \ref{convcor}) to silver polynomials and their companion matrices as partition matrices. We first note some observations.

 \begin{remark}
    Since the partition matrix $U$ is the companion matrix \eqref{Uforsilver} of a silver polynomial, the inflation-substitution map (Definition \ref{def:ins}) and the iteration (Definition \ref{def:inscon}) are determined by the indices
    \[
    1\leq s_1<\cdots<s_r=N
    \]
    with $b_{s_i}=1$, and a permutation $\pi$ of $(s_1,\ldots,s_r)$ such that 
    \[
    {\rm ins} (R_1)=(R_{s_{\pi(1)}},\ldots,R_{s_{\pi(r)}}),
    \]
    since the initial strings of inflated $R_2,\ldots,R_N$ are uniquely determined by the shape of $U$. Moreover, by Lemma \ref{joinlem} only the first index $s_{\pi(1)}$ in the initial string determines the convergence properties of the sequence of iterations, as well as the limit (if it exists) of any subsequence.
 \end{remark}

\begin{proposition}\label{divrem} Consider a primitive silver polynomial, with companion matrix $U$ as in \eqref{Uforsilver}.
\begin{enumerate}
\item If $b_1=1$, then by the inflation-substitution iteration one obtains a convergent sequence, with its limit a $\rho$-inflationary tiling. 
\item Now assume that $b_1=0$.
\begin{itemize}
    \item For any $\ell>0$ such that $b_\ell=1$, one obtains an inflationary tiling with multiplier $\rho^\ell$, by the inflation-substitution iteration and passing to a suitable convergent subsequence.
    \item Generally, by the inflation-substitution iteration one obtains inflationary tilings with minimal multiplier $\rho^m$, if and only if $b_m=1$.
    \end{itemize}
\end{enumerate}
\end{proposition}

\begin{proof} The proof essentially amounts to an application of Corollary \ref{convcor}. We use some terminology introduced in the Appendix, Subsections \ref{subsec:nonneg} and \ref{convergence}.
\begin{enumerate}
\item If $b_1=1$, then there exist substitution rules satisfying $\rho R_1=R_1|\cdots$, and any such substitution rule yields a convergent sequence, with its limit a $\rho$-inflationary tiling.
\item Now assume that $b_1=0$. By Lemma \ref{joinlem} it suffices to consider initial strings $(R_j)$, with $\mathcal X=(R_j,E,\ldots)$.
 By the shape of the companion matrix, there exists no index $j$ such that $\rho R_j=R_j|\cdots$; therefore none of the sequences constructed via Definition \ref{def:inscon} converges. 
 Now consider the initial string $(R_j)$. Then there exists a smallest integer $k>1$ such that $\rho^k R_j=R_j|\cdots$. The sequence $(\mathcal X_{1+k\ell})_{\ell\geq 0}$
    converges, and the limit is inflationary with multiplier $\rho^k$. \\
Thus for any $\ell>0$ such that $b_\ell=1$, one obtains an inflationary tiling with multiplier $\rho^\ell$, by 
    any substitution rule satisfying $\rho R_1=R_\ell|\cdots$, and passing to the subsequence $(\mathcal X_{1+\ell m})_{m\geq 0}$.
    To see the remaining assertion, since $\rho R_j=R_{j-1}$ when $j>1$, one may furthermore assume $j=1$. Now assume $b_\ell=1$ and a decomposition $\rho R_1=R_\ell|\cdots$. Then $\rho^k\,R_1=R_{\ell+1-k}|\cdots$ for $k<\ell$, and $\rho^\ell\,R_1=R_1|\cdots$. The assertion follows.
\end{enumerate}
\end{proof}

\begin{remark} \label{lenghthoftiles}
    In the case with $b_1=1$, according to Section \ref{subsec:nonneg}, Definitions \ref{def:ins} and \ref{def:inscon}, one can construct inflationary tilings of the half-line from the primitive silver number $\sigma$, with prototiles of lengths
\[
1,\,\frac{1}{\sigma},\ldots,\frac{1}{\sigma^{N-1}}
\]
and multiplier $\sigma$. 
For the case $N=2$ there is just one tiling of this type; for $N>2$ the freedom of choice for arranging prototiles in inflated tiles provides more options for the inflation-substitution construction. 
But in the primitive cases with $b_1=0$, this construction cannot produce $\sigma$-inflationary tilings. 
\end{remark}

\begin{remark} \label{aboutrhohochk}
\begin{enumerate}
\item We briefly discuss tilings obtainable in the non-primitive case from iteration of inflation and substitution. Thus let $d>1$ and $P(X) = Q(X^d)$, with primitive companion matrix $U_Q$ for $Q$. The companion matrix $U_P$ for $P$ is irreducible; see e.g. \cite{HoJo}, Theorem 6.2.24, part (c).

Writing
$Q(X)= X^N-\sum b_j X^{N-j}$, the $Nd\times Nd$ companion matrix $U_P$ of $P$, has in the first row non-vanishing entries $b_j$ at most  at positions $(1,d\cdot j)$. The inflation-substitution iteration according to subsection \ref{subsec:nonneg} is applicable to $U_P$: the matrix  $U_P$ is irreducible, as noted just above, and  Lemma \ref{reslem} implies that the largest real root $\sigma$ of $P(X)$ satisfies $1 < \sigma < 2$. By the shape \eqref{Uforsilver} of the companion matrix, it admits the eigenvector  $(\sigma^{Nd-1},\ldots,\sigma,1)^{\rm tr}$, thus the length of prototile $j$ equals $L_0\cdot \sigma^{-j}$, $1\leq j\leq Nd-1$, with  some positive factor $L_0$.
Let us now  consider any iteration. Since $U_P$ is a companion matrix of a silver polynomial we obtain $\rho R_j = R_{j-1}$
for all prototiles of index $j = 2, \dots N.$ Finally, we read off from
$U_P$ the substitution $\rho R_1 = R_{\pi(1)}|R_{\pi(2)} \dots,$ where $\pi$ is some a permutation of the numbers $1, \dots N,$  that restricts to a permutation of those prototiles $R_{sd +1}$ for which 
$b_{ls + 1}^{(P)} \neq 0$ holds. Since we consider a silver polynomial $P$,
the condition $b^{(P)}_{ls + 1} \neq 0$ actually means $b^{(P)}_{ls + 1} = 1.$
Therefore, the only free choice in this inflation-substitution construction is the restriction of $\pi$ to the set of non-vanishing integers  $ds + 1$ satisfying $b^{(P)}_{ds + 1} \neq 0.$ 
Finally, let us consider $U_Q$. It is the companion matrix of $Q$ and has non-vanishing coefficients $b^{(Q)}_k$ in the first row only for $k = s +1$,
if $b^{(P)}_{ds+1} \neq 0.$ 
Reading the restriction of $\pi $ above as a permutation $\pi^{(Q)}$ of the set of these integers $s+1,$ we obtain, from an inflationary tiling $\mathcal{T}^{(P)}$ for $P$ constructed with $U_P$ and $\pi$, an inflationary tiling $\mathcal{T}^{(Q)} $for $Q$ constructed with $U_Q$ and $\pi^{(Q)}.$ 
\item It is easy to verify that  
from a tiling $\mathcal{T}^{(Q)}$ one can construct conversely a tiling 
$\mathcal{T}^{(P)}$.
\item Theorem \ref{classifyprimitive} and the construction outlined in this remark can be generalized to irreducible non-negative integer matrices with spectral radius $\rho >1.$ 

\end{enumerate}
\end{remark}

\subsection{The golden number} \label{goldennumber}

{Since} the golden number $ \rho_2 = \phi$ is a distinguished silver number, the results of the previous section apply, and we have a $\phi$-inflationary tiling from Proposition \ref{inttile}, with prototiles of length $1$ and $1/\phi$. On the other hand, one obtains a tiling via inflation and substitution, with the same prototiles. Resnikoff \cite{Res} observed that these two tilings are identical, and his discussion of $\phi$-integers is our motivation for considering more general silver numbers and silver integers.

The fact that these two tilings coincide seems of some interest, because the representation via $\phi$-integers provides a ``closed form'' representation of the endpoints. We will give a detailed proof in the following, and we will generalize the statement and the proof to distinguished silver numbers in the next subsection.

We recall that the 
normal form representation  of a nonzero $\phi$-integer satisfies
\begin{equation}\label{eqphiint}
x=\phi^m+\sum_{j=1}^m c_i\phi^{m-i}; \quad c_i c_{i+1}=0 \text{  for all  }i\geq 0
\end{equation}
with some uniquely determined non-negative integer $m$.

Next we  specialize Theorem \ref{rhotiles}:
\begin{corollary}\label{finedistlem}
    Let the normal form \eqref{eqphiint} of $x$ be given. Then the next largest $\phi$-integer following $x$ is equal to $x+1$ if $c_m=0$, and equal to $x+1/\phi=x+\phi-1$ if $c_m=1$.
\end{corollary}

Now we state the main result for tilings from golden integers.
\begin{proposition}\label{goldeninttileprop}
    Let $x_0=0$ and denote by $x_1,x_2,\ldots$ the ordered sequence of non-zero $\phi$-integers. 
    \begin{enumerate}[(a)]
        \item The intervals $[x_{k-1},x_k]$ form a $\phi$-inflationary tiling of the half-line $[0,\infty)$, with prototiles of lengths $1$ resp. \ $1/\phi$.
        \item This tiling is equal to the one obtained via Definitions \ref{def:ins} and \ref{def:inscon}, with the partition matrix $U=\begin{pmatrix}
            1&1\\ 1&0
        \end{pmatrix}$ and initial tile $R_1=[0,\,1]$.
    \end{enumerate}
\end{proposition}
\begin{proof} 
    Part (a) is a special case of Proposition \ref{inttile}. Turning to part (b), consider the interval $J=[x_k,x_{k+1}]$ and its image $\phi J=[\phi x_k,\phi x_{k+1}]$. If $J$ has length $1$, then $\phi x_k+1$ is the next largest integer following $\phi x_k$, followed in turn by $\phi x_{k+1}$. So the inflated interval is the union of a prototile of length $1$ and a prototile of length $1/\phi$, in this order. If $J$ has length $1/\phi$, then $\phi x_{k+1}$ is the next largest integer to $\phi x_k$. So the inflated interval is a prototile of length $1$. Given the starting interval $[0,1]$, these observations correspond to the inflation and substitution rules stated in Definitions \ref{def:ins} and \ref{def:inscon}.
\end{proof}

\begin{corollary}
For the golden number $\phi$ the corresponding $\phi-$integers form 
one of the tilings obtained by the inflation and substitution rules 
stated in  Definitions \ref{def:ins} and \ref{def:inscon}.
\end{corollary} 
\subsection{Distinguished silver numbers of higher degree}

Here one cannot expect an exact replica of Proposition \ref{goldeninttileprop}, since the prototiles in both constructions  are not identical: In the integer case we have lengths
\begin{equation*}
1, \quad \dfrac{1}{\rho}+\cdots+\dfrac{1}{\rho^{N-1}}, \quad \dfrac{1}{\rho}+\cdots+\dfrac{1}{\rho^{N-2}}, \quad \ldots \quad \dfrac{1}{\rho},
\end{equation*}
while in the inflation-substitution case with partition matrix \eqref{Uforsilver} we have lengths
\[
1,\quad\dfrac{1}{\rho},\quad\ldots\quad \dfrac{1}{\rho^{N-1}}.
\]

But one obtains:
\begin{theorem}\label{silvintprop}Let $\rho_N = \rho$ be the distinguished silver number of degree $N\geq 3$.
\begin{enumerate}
    \item Consider the inflation-substitution tiling $\mathcal T$ with prototiles $R_j$ of length $\rho^{1-j}$, $1\leq j\leq N$, with starting interval $[0,1]$ and substitution rules
\[
\rho R_1= R_1|R_2|\cdots|R_{N-1}|R_N;\quad \rho R_j= R_{j-1},\,2\leq j\leq N;
\]
equivalently constructed from the companion matrix in (\ref{Uforsilver}), but  with all $b_i=1$.
Then the tiling of the half-line by $\rho$-integers is subordinate to $\mathcal T$; i.e., every $\rho$-integer is the endpoint of some tile in $\mathcal T$.
\item On the other hand, consider the prototiles $\widehat R_j$ of lengths, respectively, 
\[
1,\quad \dfrac{1}{\rho}+\cdots+\dfrac{1}{\rho^{N-1}}, \quad \dfrac{1}{\rho}+\cdots+\dfrac{1}{\rho^{N-2}}, \quad \ldots \quad \dfrac{1}{\rho},
\]
with starting interval $[0,1]$ and substitution rules
\[
\rho \widehat R_j= \widehat R_1|\widehat R_{j+1},\, 1\leq j\leq N-1,\quad \rho \widehat R_N= \widehat R_{1}.
\]
With these one obtains a $\rho$-inflationary tiling $\widehat{\mathcal T}$, which may actually be constructed from the transpose of the companion matrix from \eqref{Uforsilver}; see Remark \ref{companiontypes}, equation \eqref{CDWtranpose}, with all $b_i=1$.
The endpoints of this tiling are exactly the $\rho$-integers. 
\end{enumerate}
\end{theorem}
\begin{proof} Let 
\[
0=x_0<x_1<x_2<\cdots
\]
be the ordered sequence of $\rho$-integers.  According to Theorem \ref{rhotiles}, the differences of consecutive integers are given as follows:
\begin{itemize}
\item Case 1: $x_{k+1}=x_k+1$, so the interval $[x_k,x_{k+1}] $ will be inflated to 
\[
[\rho x_k,\rho x_{k+1}]=[\rho x_k,\rho x_k+1+\dfrac1{\rho}+\cdots+\dfrac{1}{\rho^{N-1}}].
\]
This corresponds to the substitution $\rho R_1= R_1|R_2|\cdots|R_{N-1}|R_N$, respectively $\rho \widehat R_1=\widehat R_1|\widehat R_2$.
\item Case j ($j\geq 2$): Here we have 
\[
x_{k+1}=x_k+\dfrac1{\rho}+\cdots+\dfrac{1}{\rho^{N-j}};\quad \rho x_{k+1}=\rho x_k+1+ \dfrac1{\rho}+\cdots+\dfrac{1}{\rho^{N-j-1}}.
\]
This corresponds to the substitutions 
\[
\rho\left(R_2|\cdots|R_{N-j+1}\right)=R_1|\cdots|R_{N-j},
\]
which are induced by
$\rho R_\ell= R_{\ell-1}$, $\ell>1$; respectively the substitutions 
\[
\rho\widehat R_k=\widehat R_1|\widehat R_{k+1},\, k\leq N-1;\quad \rho \widehat R_N=\widehat R_1.
\]
\end{itemize}
\end{proof}

\begin{corollary} \label{rhointinfl}
Let $\rho$ be a distinguished silver number. Then the
 tiling $\mathcal{T}_\rho$ with the set of $\rho-$integers as endpoints is
not periodic.
\end{corollary}
\begin{proof}
By the second part of the theorem, $\mathcal{T}_\rho$ is also obtainable from an inflation-substitution iteration with a primitive partition matrix with spectral radius $\rho$. Now the assertion follows directly from 
 Theorem \ref{penrose}.
 \end{proof}   
 
 \section{About non-distinguished silver numbers}\label{sec:nondis}

\subsection{Non-distinguished silver numbers of degree three}

As noted in Example \ref{exsilpo}, silver numbers of degree three correspond precisely to the polynomials
\[X^3 - X^2 - X - 1, \quad X^3 - X^2  - 1, \quad X^3 - X - 1.\] 
The first defines the distinguished silver number $\rho_3$ of degree three, which is also called the ``tribonacci'' number. The second defines the ``supergolden" number $\psi$ and the third polynomial defines the ``plastic number" $\theta$. An application of Corollary \ref{simple} shows that the partition matrices of all these silver numbers are primitive.

With regard to inflation-substitution tilings with inflationary factor $\psi$ and $\theta$ respectively, corresponding to these silver numbers, Proposition \ref{convcor} and Remark \ref{divrem} are readily applicable. In particular the distinguished silver number of degree three yields  an inflationary tiling with multiplier $\rho_3$.\\
We will discuss the  two non-distinguished silver numbers of degree three, and tilings related to them, as follows.
\subsubsection{The super-golden number.} The super-golden number $\psi$ is the largest zero of the irreducible polynomial $X^3 - X^2 -1$ ; see Lin \cite{Lin} for more of its properties. The partition matrix \eqref{Uforsilver} specializes to

\begin{equation} \label{Usupergolden}
U=\begin{pmatrix} 1&0&1\\ 
                               1 & 0& 0\\
                                0&1 & 0\\
    \end{pmatrix}.
\end{equation}

According to Example \ref{instileex}(b), one obtains a $\psi$-inflationary tiling from the inflation-substitution iteration, as well as a  $\psi^3$-inflationary tiling.
We observe  that $\psi^2$ is not a silver number: note that $\psi>1.46$, hence $\psi^2>2$, and use Lemma \ref{reslem}.\\
       There remains the question whether the $\psi$-integers are subordinate to some tiling with prototiles of lengths $L$, $L/\psi$ and $L/\psi^2$, for some basic length $L>0$ (not necessarily $L=1$). We show that this is not the case for any rational $L$.\\
    Assume the contrary. Since $\psi+1$ and $\psi^2$ are $\psi$-integers, their difference would be a non-negative integer combination of the prototile lengths, thus 
    \begin{equation*}
        \psi^2-\psi-1=La+Lb/\psi+Lc/\psi^2;
    \end{equation*}
    equivalently
    \begin{equation*}
       \psi^4- \psi^3-\psi^2=La\psi^2+Lb\psi+Lc.
    \end{equation*}
    Noting that $\psi^4-\psi^3=\psi$, and combining this with the defining identity for $\psi$, one finds
    \begin{equation*}
        (La+1)\psi^2+(Lb-1)\psi+Lc=0,
    \end{equation*}
    with $La+1>0$. This implies that $\psi$ is a root of a quadratic polynomial with rational coefficients. Since the polynomial $x^3-x^2-1$ is irreducible over the rationals (otherwise it would have a degree one factor and a rational root, which is excluded by Lemma \ref{reslem}), it is the minimal polynomial of $\psi$, and we arrive at a contradiction. But note that this argument does not exclude possible tilings with irrational $L$.
\subsubsection{The plastic number} The {\em plastic number} $\theta$ is the largest zero of the irreducible polynomial $X^3 - X -1$; see van der Laan \cite{Laan}, Shannon et al. \cite{SAH} for more properties. Here we have the partition matrix
\begin{equation} \label{Usupergolden}
U=\begin{pmatrix} 0&1&1\\ 
                               1 & 0& 0\\
                                0&1 & 0\\
    \end{pmatrix}.
\end{equation}
According to Example \ref{instileex}(c), by the inflation-substitution procedure one cannot obtain $\theta$-inflationary tilings, but from subsequences one obtains tilings with multiplier $\theta^2$. As shown there, $\theta^2$ is not a root of an irreducible silver polynomial, but at this point we cannot exclude the possibility that $\theta^2$ is a root of a reducible silver polynomial. \\
         Moreover a variant of the argument for the super-golden number shows that the integers for the plastic number 
    are not subordinate to an inflationary tiling of the half-line, with prototiles of lengths $L$, $L/\theta$ and $L/\theta^2$ with any rational $L$:\\
    Again, assume the contrary. Since $1$ and $\theta$ are $\theta$-integers, their difference would be a non-negative integer combination of the prototile lengths, thus
    \begin{equation*}
        \theta-1=La+Lb/\theta+Lc/\theta^2;
    \end{equation*}
    equivalently
    \begin{equation*}
        \theta^3-\theta^2=La\theta^2+Lb\theta+Lc.
    \end{equation*}
    Combining this with the defining identity for $\theta$, one finds
    \begin{equation*}
        (La+1)\theta^2+(Lb-1)\theta+(Lc-1)=0,
    \end{equation*}
    with $La+1>0$. This implies that $\theta$ is a root of a quadratic polynomial with rational coefficients. Since the polynomial $x^3-x-1$ is irreducible over the rationals (otherwise it would have a degree one factor and a rational root, which is excluded by Lemma \ref{reslem}), it is the minimal polynomial of $\theta$, and we arrive at a contradiction. As in the previous example, we cannot exclude tilings with irrational basic length $L$.
\subsection{A wider perspective}
In view of  the previous results and examples, for a 
non-distinguished silver number $\rho$ of degree $N$, it would be interesting to know whether the corresponding ``tiling via integers''  (which actually may not be a tiling in the sense of Section \ref{appsection})  is subordinate to some tiling with prototiles of lengths 
$L, L/{\rho}, \dots , {L/\rho^{N-1}}$, for some basic length $L>0$. To rephrase the question: Is there a tiling with these prototiles such that every $\rho$-integer is the endpoint of some tile? (For the supergolden number, and for the plastic number,
we have just shown that this is not possible with rational basic length $L$.) 

If, for some $\rho$, the answer to the above question is positive, then one would have a  possibly new method to construct inflationary tilings. And in particular, a positive answer for the plastic number $\theta$, would show that a $\theta$-inflationary tiling exists.\\
For a positive answer one certainly requires a positive lower bound for the absolute values of differences of $\rho$-integers, considering hypothetical prototiles and taking the one with smallest length (compare also Corollary \ref{distnoclus}). 

To some extent, this necessary condition is also sufficient. We can show:
\begin{theorem}\label{thm:nondist}
    Let $P=X^N-\sum b_iX^{N-i}$ be an irreducible silver polynomial of degree $N$, and assume furthermore that the corresponding silver number, $\rho$, is a Pisot number. Then the following are equivalent.
    \begin{enumerate}
        \item There exists $L>0$ such that the $\rho$-integers form the endpoints of a tiling that is subordinate to a tiling with prototiles of lengths $L$, $L/\rho$,$\ldots$,$L/\rho^{N-1}$.
        \item The set of differences of $\rho$-integers does not have cluster point $0$.
    \end{enumerate}
\end{theorem}
\begin{proof} We need to prove the nontrivial implication 
$(2) \Rightarrow  (1)$ only. Setting

\[
{\mathcal J}:=\left\{ \sum \delta_iX^i\in \mathbb R[X];\quad \delta_i\in \{-1,0,1\}\right\},
\]
the differences of $\rho$-integers are given by the $q(\rho)$, with $q\in \mathcal J$. Our hypothesis then says that 
\begin{equation}\label{mudef}
\mu:={\rm inf}\,\left\{ |q(\rho)|;\, q\in \mathcal J\text{  and  } q(\rho)\not=0\right\}>0.
\end{equation}
We need to prove the existence of an $L>0$ such that for every $q\in\mathcal J$ there exist integers $m_1,\ldots, m_N$, {\em all of the same sign}, such that
\[
q(\rho)=\sum_{i=1}^N m_iL/\rho^{i-1}.
\]
\begin{enumerate}[(i)]
\item We first assume that such an $L>0$ exists and start with a representation of $1$ with non-negative integers $a_{j,0},$  $j = 0, \dots  , N-1,$ satisfying
\begin{equation} \label{equ0}
1 = L(a_{0,0} + \frac{a_{1,0}}{\rho} +  \frac{a_{2,0}}{\rho^2} + \dots  +
\frac{a_{N-1,0}}{\rho^{N-1}})
\end{equation} 
To obtain a representation for $\rho$ we multiply $(\ref{equ0})$
by $\rho$ and substitute 
$\rho = b_1 + \frac{b_2}{\rho^{1}} +  \dots + \frac{b_{N-1}}{\rho^{N-2}}+   \frac{b_N}{\rho^{N-1} }$ from the 
equation $P(\rho ) =0.$ We find
\begin{equation} \label{equ1}
\begin{array}{rcl}
\rho &=&L(a_{0,1} + \frac{a_{1,1}}{\rho} +  \frac{a_{2,1}}{\rho^2} + \dots  + \frac{a_{N-1,1}}{\rho^{N-1}}) \\
&=& L(a_{0,0}b_1 + a_{1,0} + \frac{a_{0,0}b_2 + a_{2,0}}{\rho} +  \frac{a_{0,0}b_3 + a_{3,0}}{\rho^2} + \dots  + 
 \frac{a_{0,0}b_N + a_{N-1,0}}{\rho^{N-1}})
 \end{array}
\end{equation} 
Since $P$ is irreducible, the set of powers $\rho^{-j}$, $0\leq j<N$, is linearly independent over $\mathbb Q$, hence one can compare coefficients at the powers $\rho^{-j}$ term by term. It is straightforward to verify that the vector of coefficients $v_1$  representing $\rho = \rho^1$ is related to the (original) vector of coefficients  
\[
v_0=(a_{0,0},\, a_{1,0},\,\ldots, \,a_{N-1,0})^{\rm tr}
\]
which represents $1 = \rho^0$ by multiplication with $U^{tr}$, where $U$ is the companion matrix of $P(X).$ 
Iterating multiplication  by $\rho$ and collecting terms one obtains for $A = U^{tr}$ all $k = 0,1, \dots :$
\begin{equation}\label{vrecursion}
v_{k}=A  = A^k\, v_0.
\end{equation}
Note that $\rho$ and $U$ and $A$ all have the same minimum polynomial, $P(X).$ Moreover, $U$ and $A$ have
the same characteristic polynomial and are both primitive (in view of the Pisot property and Theorem \ref{primmat}).
The strategy for the proof of the theorem will now be to find a non-negative integer vector $v_0$ such that all entries $q(A)v_0$ have the same sign as $q(\rho)$. Then one defines $L$ by equation $(\ref{equ0})$
and writes 
\begin{equation}
q(\rho)= L(1, 1/\rho, \dots , 1/ \rho^{N-1}) q(A)v_0,
\end{equation}
obtaining the required result.
\item  Since $A$ is primitive, by the last item of Corollary \ref{CortoPrim} we can decompose $\mathbb{R}^N =  \mathbb{R} w \oplus w^\perp$, where $w$ is a positive eigenvector for $\rho$ of length $1$ and  $Z = w^\perp$ is the kernel of the linear map given by $w^{tr}$ on $ \mathbb{R}^N.$ 
By the Pisot property there exists $0<\delta <1$ with
\[
\Vert Az\Vert \leq \delta \Vert z\Vert; \quad \text{all  }z\in Z.
\]
Now consider some 
\[
v=\alpha w + z;\quad \alpha\in \mathbb R, \quad z\in Z.
\]
Then for any $q\in\mathcal J$ we have 
\[
q(A)\, v=q(\rho)\alpha \,w + q(A)\,z.
\]
For all $q$ we have the estimate
\[
\Vert q(A)\,z\Vert\leq \sum_{k\geq 0}\Vert \delta^k A^k z\Vert\leq \sum_{k\geq 0}\delta^k\Vert z\Vert=\gamma \Vert z\Vert;\quad \gamma=1/(1-\delta).
\]
Now assume $\mu>0$ (see \eqref{mudef}). Let $\omega>0$ be the smallest entry of the eigenvector $w$, and  consider $v$ with $\alpha>\frac12$. Then for $q(\rho)\not=0$, every entry of $q(A)v$ is
\[
\begin{array}{rccl}
\geq &\alpha \mu \omega-\gamma \Vert z\Vert&\text{if}&q(\rho)>0;\\
\leq &-\alpha \mu \omega+\gamma \Vert z\Vert&\text{if}&q(\rho)<0.
\end{array}
\]
\item 
So, for all $\alpha>\frac12$ and all $z\in Z$ with $\Vert z\Vert <\frac12 \mu\omega/\gamma$, we have:
Whenever $v=\alpha w+z$, and $q\in \mathcal J$, $q(\rho)\not=0$, then all entries of $q(A)v$ have the same sign, which equals the sign of $q(\rho)$. Since the conditions on $v$ define an open subset of $\mathbb R^N$, there exists a rational vector $\widetilde v$ in this set. Taking an integer multiple $v_0$ of $\widetilde v$, we have found an integer vector such that all entries of $q(A)v^*$ have the same sign, as desired. Finally, $L$ can be obtained from \eqref{equ0} or any of the equations for $\rho^k.$
\end{enumerate}
\end{proof}

While the Theorem just above does not provide a definitive result for silver integers (except the distinguished ones, for which the existence of a tiling was shown directly), it provides a \\

{\bf Dichotomy}:
For a silver number satisfying the hypotheses of Theorem \ref{thm:nondist}, we have: Either the set of differences of pairs of $\rho$-integers has $0$ as cluster point, or the $\rho$-integers form endpoints of a $\rho$-inflationary tiling.

\begin{remark} The latter case would seem of particular interest in settings (e.g. the plastic number) when the inflation-substitution iteration does not yield a tiling with inflation multiplier $\rho$.
\end{remark}

\section{Appendix: Inflation-substitution tilings}\label{appsection}
Here  we collect notions and facts about substitution tilings of the real half-line, for easy reference and the readers' convenience. Apart from some (elementary, easy-access) constructions and the mode of presentation, the authors do not claim that this section contains original results.

Throughout we consider subsets of $\mathbb R$. Tilings of $\mathbb R$ and of the half-line $[0,\infty)$ have been thoroughly discussed in the literature from a dynamical systems perspective; we mention in particular Barge and Diamond \cite{BarDia}, and the references therein.\\
Some general concepts and notions for tilings may be simplified in the setting with intervals, and we will use this fact in our discussions.
\subsection{Definitions}
\begin{definition} \label{def:basicdef1d} 
\begin{enumerate}[1.]
    \item A {\em tile} in $\mathbb R$ is a non-trivial compact interval.

\item Two tiles are said to be  
essentially disjoint, if their intersection is either empty or consists of one point only.

\item  Two tiles are said to be {\em equivalent} if they have equal length, thus are translates of each other.
\end{enumerate}
\end{definition}
\begin{definition}
Given any closed interval $J$ with non-empty interior, and some index set $\mathcal{I}$, we say that a set $\mathcal T = \{T_j\}_{j \in \mathcal{I}}$ of tiles is a {\em tiling} (or {\em tesselation}) of $J$ if:
  \begin{enumerate}
  \item Any two tiles of $\mathcal{T}$ are essentially disjoint.
\item $J = \cup_{j \in \mathcal{I}} T_j.$
In particular, any tile is a subset of $J$.
\item There only occur finitely many lengths among the tiles in  $\mathcal{T}.$
  \end{enumerate}
  \end{definition}

  \begin{definition}
If $L_1, \dots L_M$ are all the pairwise different lengths of tiles appearing in a tesselation $\mathcal{T},$
then any choice {$\mathcal{R} = 
\{R_1, \dots, R_M\}$ of representatives} of the tiles of lengths $L_1,\ldots,L_M$ respectively is called a set of {\em prototiles}.
{In particular, any tile of a given tesselation is a translate of exactly one prototile in the chosen set $\mathcal{R}$.}
\end{definition}
  
\begin{remark} 
\begin{enumerate}
\item Any two sets of prototiles of a tiling are naturally in bijection.
\item A tiling of a compact interval $J$ necessarily consists of finitely many tiles.
\item Since tiles are non-trivial intervals, their location along the real line is uniquely determined.
Therefore, a tiling ${\mathcal T}$ of $[0,\infty)$ is equivalently characterized by a strictly increasing sequence $0=y_0<y_1<\cdots$ of real numbers such that $[0,\infty)$ is the essentially disjoint union of the intervals $S_i:=[y_{i-1},y_i]$, ${i = 1, 2, \dots}$. The $S_i$ are then the tiles of the tesselation, and we will call
the  points  {$y_{i-1},\,i = 1, 2, \dots$, the endpoints of the tiling. We also write $\mathcal T=(S_i)_{i\in \mathbb N}$.  }
\end{enumerate}
\end{remark}

\begin{definition}
\begin{enumerate}    
        \item We say that a tiling ${\mathcal T}'$ is {\em subordinate} to the tiling $\mathcal T$ if every endpoint of ${\mathcal T}'$ is also an endpoint of $\mathcal T$. 
    
\item {Subdividing a non-trivial compact subinterval $[a,b] \subset [0, \infty)$ by finitely many points $y_0 = a < y_1 <  \dots < y_l = b,$ we can define analogously to the above notion $[a,b]=(S_i)_{k\leq \leq \ell}$  with $S_i:=[y_{i-1},y_i], i = 1, 2, \dots l.$}
    \end{enumerate} 
\end{definition}
For later use we also note the 
\begin{definition} \label{defperiodic}
Using the notions introduced above, a tiling $(S_i)_{i\in \mathbb N}$ of $[0,\infty)$ is called {\em periodic} if there is an $m\in\mathbb N$ such that $S_i$ and $S_{i+m}$ are equivalent for all $i$. A tiling is called {\em ultimately (or eventually) periodic} if there exist $m\in\mathbb N$ and $j\in\mathbb N$ such that $S_i$ and $S_{i+m}$ are equivalent for all $i>j$.
\end{definition}
The following definitions are based on equivalent, or more general, notions in Resnikoff \cite{Res}, p.~6, and Barge and Diamond \cite{BarDia}, sections 1 and 2.
\begin{definition}\label{inflatess}
 Consider a set of tiles $R_1,\ldots,R_N$, 
of pairwise different lengths $L_1,\ldots,L_N$ respectively. Moreover let $\rho\in\mathbb R,\,\rho>1$.
\begin{enumerate}[1.]
    \item  This set of tiles is called {\em inflationary}, with {\em multiplier} $\rho$, if every set $\rho R_i$ is the union of essentially disjoint tiles, each of which is equivalent to some $R_j$.
    \item A tiling $(S_i)_{i\in \mathbb N}$ of $[0,\infty)$ with prototiles {$\mathcal{R} = \{R_1,\dots, R_N\}$} is called {\em inflationary} with multiplier $\rho$ (or, briefly, $\rho$-inflationary), if the following holds:
    \begin{enumerate}[(i)]
        \item For each $i\in\{1,\ldots,N\}$, there is a given order for the representation of $\rho R_i$ as essentially disjoint union of tiles equivalent to the $R_j$. We furthermore stipulate that this order is extended to all $\rho S_k$, $S_k$ equivalent to $R_i$.
 \item The tiling $(\widetilde S_i)_{i\in\mathbb N}$ of $[0,\infty) =
 (\rho S_i)_{i\in\mathbb N}$ which is obtained by substituting for each $\rho S_i$ the essentially disjoint union from item (i), is equal to $(S_i)_{i\in \mathbb N}$.
    \end{enumerate}
\end{enumerate}
\end{definition}

\begin{remark} \label{typeofinflation}
Item $(ii)$ just above fixes the notion of  ''inflationary tiling" in a very specific way. In particular, the first appearing prototile, call it w.l.o.g. $R_1$, is inflated necessarily of the form 
$\rho R_1 = R_1| \dots |R_\star$. 
However, at various places we are considering more general situations. See e.g. Remark \ref{aboutrhohochk},  Corollary \ref{convcor}, and Corollary \ref{divrem} for more details. 

\end{remark}
\begin{remark}
    If a set of prototiles is inflationary with multiplier $\rho$, then obviously it is also inflationary with any multiplier $\rho^m$, $m$ a positive integer. Likewise, if a tiling is inflationary with multiplier $\rho$, then it is also $\rho^m$-inflationary for all positive integers $m$. On the other hand, there exist examples (e.g. the plastic number $\theta$; see Example \ref{instileex}) of $\rho$-inflationary sets of prototiles that do not admit a $\rho$-inflationary tiling by the inflation-substitution procedure, since item $2.(ii)$ of the definition just above is not satisfied. But there 
    exists a $\rho^m$-inflationary tiling for some $m>1$. See 
    Corollary \ref{convcor} and Corollary \ref{divrem} for more details.
\end{remark}

\begin{remark}
\begin{enumerate}
     \item If the lengths $L_j$ 
     of the prototiles $R_j$ are linearly independent over the rationals, then the union in item (i) of the definition is unique up to the ordering of the prototiles, since the length of $\rho R_i$ is a (non-negative) integer linear combination of the $L_j$.
     In general, the choice of a partition matrix for an inflationary tiling makes the union in item (i) above unique up to ordering.
    \item From the definition above it is clear that $\rho = 1$ 
    does not produce any inflation and thus no inflationary tilings. 
     Clearly, the case $\rho < 1$ does not make sense in the case of inflationary tilings, since then one would need infinitely many different prototiles.
\end{enumerate}
\end{remark}

We extend the notion ``$\rho$-inflationary'' to sequences, as follows.
\begin{definition} Let $0=x_0<x_1<\cdots$ be a strictly increasing sequence, and $\rho>1$. Then we call the sequence {\em $\rho$-inflationary} if $\rho x_j\in\left\{x_i;\, i\in \mathbb N_0\right\}$ for all $i\geq 0$.
\end{definition}
\begin{remark}
    The endpoint sequence of a $\rho$-inflationary tiling is obviously $\rho$-inflationary. But clearly there are $\rho$-inflationary sequences (e.g. take $x_i=\rho^i$, $i>0$) that cannot be obtained from inflationary tilings according to Definition \ref{inflatess}. 
\end{remark}

In the following sections we will discuss the representation of tilings and the construction of inflationary tilings. 

\subsection{Representing tilings of intervals}\label{sec:1Dsequences}

We consider
primarily the interval $[0,\infty)$; but compact subintervals will be included into the discussion as well.
\subsubsection{\bf Indicator sequences}

\begin{remark}\label{indicatorrem}
\begin{enumerate}
    \item Let $\mathcal T$ be a tiling of $[0,\,\infty)$ together with a fixed choice  $\mathcal{R} = \{R_1,\ldots,R_N \}$  of prototiles, and moreover let $0=y_0<y_1<\cdots$ be the sequence of endpoints, with $S_i=[y_{i-1},\,y_i]$, $i\geq 1$. Then for each $i\geq 1$ there is a unique $j(i)\in\left\{1,\ldots,N\right\}$ such that $S_i$ is equivalent to $R_{j(i)}$.\\
   Therefore the tiling is uniquely characterized by the sequence $(R_{j(i)})_{i\in\mathbb N}$. 
   \item An analogous statement holds for tilings  of non-trivial compact intervals, corresponding to a finite sequence (also called a {\em finite string}) of prototiles.
  \item To put tilings of finite intervals $[0,\,a]$ on the same footing as  tilings of the positive real half-line, we introduce an additional ``prototile" $E\not\in \{R_1,\ldots, R_N\}$ (standing for ``empty space''), and complete the finite string $(R_{j(1)},\ldots R_{j(m)})$ which represents the tiling of $[0,\,a]$ to the infinite sequence 
  \[
  (R_{j(1)},\ldots R_{j(m)}, E,\ldots).
  \]
  Geometrically, one may think of finite tilings as incomplete tilings of $[0,\infty)$.
    \end{enumerate}
      \end{remark}

{\em From here on we will frequently represent tilings as infinite sequences of prototiles.}
 
\begin{definition}
We will call the  sequences constructed just above 
{\em indicator sequences} of the respective tilings.
\end{definition}

\subsubsection{\bf An ultra-metric on the space of tilings of $[0,\infty)$} 
 We consider tilings with a given set 
$\mathcal{R} = \{R_1,\ldots,R_N\}$ of prototiles, and their indicator sequences.
Generalizing the setting introduced 
in the previous subsection we will now consider the space 
$\mathcal{G}(R_1,R_2, \dots , R_N, E)$ of indicator sequences with entries
from $\{R_1, \dots R_N\} \cup \{E \} $.

\begin{definition} \label{ultrametric}
For two different
sequences  $\mathcal{X} = (X_k)_{k \in \mathbb N}$,  
$\mathcal{Y} = (Y_k)_{k \in \mathbb N} \in \mathcal{G}(R_1,R_2, \dots , R_N, E)$ 
we define their distance
\begin{equation}\label{def:dist}
    dist(\mathcal{X}, \mathcal{Y}) = 
dist( (X_k),(Y_k)):=2^{-\min\{j:\,X_j\not=Y_j\}}, 
\end{equation}
and furthermore we set $dist(\mathcal{X}, \mathcal{X}) =0$.
The definition implies directly that two different indicator sequences have a non-zero distance from each other.
Equivalently,  this means
$dist(\mathcal{X}, \mathcal{Y}) =0$ if and only if $\mathcal{X} = \mathcal{Y}.$ 
\end{definition}
Moreover, $dist(\mathcal{X}, \mathcal{Y}) =
dist(\mathcal{Y}, \mathcal{X})$ for all sequences . Finally, the distance function $dist$ satisfies the  strengthened triangle inequality
\begin{equation}
dist(\mathcal{X}, \mathcal{Y}) \leq 
\max\{dist(\mathcal{X}, \mathcal{Z}),dist(\mathcal{Z}, \mathcal{Y})\},
\end{equation}
for all sequences  $\mathcal{X}, \mathcal{Y} \in \mathcal{G}(R_1,R_2, \dots , R_N, E)$.  In view of this last property we obtain

\begin{proposition}
\begin{enumerate}[a)]
    \item The set $\mathcal{G}(R_1,R_2, \dots , R_N, E)$ of indicator sequences with entries in $\{R_1,\ldots,R_N,E\}$, together
with the  distance function defined  by \eqref{def:dist} is an ultra-metric space (see e.g. \cite{Semmes}). 
\item A sequence $\left(\mathcal X_\ell\right)_{\ell\in\mathbb N}$ in this space is a Cauchy sequence  if and only if for every positive integer $k$ there exists an index $\ell_k$ such that the first $k$ entries of all $\mathcal X_\ell$ with $\ell>\ell_k$ coincide.
\item Thus every Cauchy sequence is convergent, and the space is complete.
\end{enumerate}
\end{proposition}

\begin{remark}\label{seqrem}
By the correspondence between tilings and their indicator sequences we obtain a metric on the set of tilings. This gives a well-defined meaning to statements like ``a sequence of tilings of compact intervals converges to a tiling of $[0,\,\infty)$''. 
\end{remark}
\subsection{Partition matrices, inflationary tilings, and a brief survey on Perron-Frobenius Theory}
Our goal is to construct inflationary tilings of the half-line $[0,\infty)$, in subsection \ref{subsec:nonneg} below. In the present section we discuss the setting and recall basic properties of non-negative matrices which will be used throughout.  

\subsubsection{\bf From inflationary tilings to partition matrices}
Consider an inflationary  tiling  
$\mathcal{T}$ (see Definition \ref{inflatess}) with an inflationary set 
$\mathcal{R} = \{R_1,\ldots,R_N\}$ of prototiles and multiplier $\rho>1$. \\
Denote the length of $R_j$ by $L(R_j)$. Each $\rho \cdot R_j$ is an essentially disjoint union of tiles equivalent to suitable $R_j$.  Comparing lengths, one finds that there are non-negative integers $u_{ij}$ such that
\begin{equation}\label{partimat}
  \rho\cdot  L(R_i)=\sum_{j=1}^N u_{ij}L(R_j),\quad 1\leq i\leq N.
\end{equation} 
Thus the $N \times N-$ matrix $U:=\left(u_{ij}\right)$ is non-negative with integer entries, and $\rho$ is an eigenvalue of $U$ with eigenvector $\left(L(R_j)\right)_{1\leq j\leq N}$. Using terminology from Resnikoff \cite{Res} we call $U$ a {\em partition matrix} of the given set of prototiles.  
(Partition matrices also play a role in tilings of  higher dimensional spaces, at least  for ``book-keeping'' purposes by comparing volumes). \\
Equation \eqref{partimat} indicates the relevance of non-negative matrices for tilings of the half-line. For this reason, and last not least for easy reference, an overview of important notions and results in the theory of non-negative matrices is provided. As will be outlined in subsection \ref{subsec:nonneg}, in dimension one every irreducible non-negative matrix with integer entries and spectral radius $\rho>1$ gives 
rise to tilings.
\subsubsection{\bf Perron-Frobenius Theory for irreducible non-negative real matrices: Preparation}
The results listed in the rest of this subsection can be found e.g. in Meyer \cite{Meyer}, sections 8.2 and 8.3, or  
in Berman and Plemmons \cite{Berman}, specifically Chapter 2 (Thms.~1.4 and 1.7, Def.~1.8, Thm.~2.7, Thm.~2.20 and Cor.~2.28) and Chapter 8 and Chapter 11. Moreover references are listed explicitly, where used, if the above sources do not apply.
\begin{enumerate}[a)] 
\item  {\em Definition.}  We call a real $N\times N-$matrix {\em non-negative} if all its entries are $\geq 0$, and {\em positive} if all its entries are $>0$. In this situation we sometimes write $A\geq 0$ and  $A > 0$ respectively.
  \item {\em Irreducibility.}
    \begin{itemize}
        \item A non-negative $N\times N-$matrix $V$ is called {\em reducible}, if there exists a permutation matrix $P \in\mathbb R^{N\times N}$  such that $P^{-1}VP$ has lower or upper block triangular form.
Otherwise $V$ will be called {\em irreducible.} 
\item Since conjugation with a permutation matrix amounts to a permutation of the matrix entries, any
triangular block form of a non-negative matrix $V$ is again non-negative. Thus a more detailed discussion of reducible non-negative matrices can be done w.l.o.g. by considering non-negative lower block matrices.
\item With regard to the construction of inflationary tilings of $[0,\infty),$ one  sometimes focuses attention on irreducible partition matrices for the sake of simplicity, due to the following observation:
\begin{lemma}
Let $U$ be reducible, w.l.o.g. $U$ is of lower block triangular form, and denote by $U'$ its upper left block of size $N' \times N'$. Then the set 
$\mathcal{R}' = \{R_1,\ldots,R_{N'}\}$ of prototiles $R_1,\ldots,R_{N'}$ of the tiling $\mathcal{T}'$ constructed from $U'$ (see below) is inflationary with multiplier $\rho$.
\end{lemma}
\end{itemize}
\item {\em Definition.}\begin{enumerate}
\item  For any real $N \times N-$matrix $A$ the set of all (possibly complex) eigenvalues is called the spectrum of $A$ and denoted by $\sigma(A)$.
\item The spectral radius of A is  given by
\begin{equation}
\rho(A) =  \max \{ |\lambda|; \lambda \in \sigma(A) \}.
\end{equation}
\end{enumerate}
\item The first of the following results is common knowledge; the second is an easy consequence. 
\begin{lemma}\label{lem:nilpo}
a) A real matrix $A$ has spectral radius $\rho(A) <1$ if and only if $A^k    \xrightarrow[k \to \infty]{}   0.$\\
b) If, moreover, $A$ is non-negative with integer entries, then: $\rho(A) < 1 \Leftrightarrow  A$ is nilpotent.
\end{lemma}

\end{enumerate}
\subsubsection{\bf Perron-Frobenius Theorem for irreducible non-negative real matrices} \label{PFgen}
\begin{theorem}\label{perfrothm}
Let $A$ be an irreducible non-negative real $N \times N-$ matrix. Then
\begin{enumerate}[a)] 
\item The spectral radius $\rho = \rho(A)$ of $A$ is a positive real eigenvalue (``Perron - Frobenius eigenvalue").
\item The eigenvalue $\rho$ is simple. In particular, the right and the left eigenspace associated with 
$\rho$ is one-dimensional.
\item The matrix  $A$ has for the eigenvalue $\rho$  both a right eigenvector $w$ and a left eigenvector $v$ 
that have only positive entries.
\item Every eigenvector of $A$ with only positive entries is associated with $\rho$.
\end{enumerate}   
\end{theorem}
\begin{definition}
Under the assumptions of the last theorem, any eigenvector for the Perron eigenvalue $\rho = \rho(A)$ of $A$ with only positive entries will be called a  ``Perron eigenvector".
\end{definition}
\subsubsection{\bf Perron-Frobenius Theorem for primitive matrices}
In the theory of non-negative matrices, a more special class of matrices $A \geq 0$ is very important.
We start by giving characterizations.
\begin{theorem}[{\bf Primitive Matrix Theorem}] \label{primmat}
The following statements for a non-negative real $N \times N-$matrix $A$ are equivalent.
\begin{enumerate}
\item  The matrix A is irreducible and the circle in the complex plane with center $0$ and radius $\rho = \rho(A)$ only contains the eigenvalue $\rho$ of $A$.
\item  The matrix $A$ is irreducible and for some $j \in \{1, \dots , N \}$ the greatest common divisor of all natural
 numbers $m$  such that $(A^m)_{jj} > 0$ is $1$ (see, e.g. \cite{Denardo} or \cite{Levi-Peres}, Lemma 1.6).
\item  There exists some $m>0$ such that $A^m$ has only positive entries.
\item The matrix A is irreducible and the limit $(\frac{1}{\rho}A)^k$ for $  \xrightarrow[k \to \infty]{} $ exists.
\end{enumerate}
\end{theorem}
\begin{definition}\label{def:primmat}
Any non-negative real matrix satisfying one (hence all) of the statements of the last theorem will be called a {\em primitive matrix}.
\end{definition}

\begin{remark}
    An obvious consequence of the theorem is that every primitive matrix is irreducible. 
    Of course, not every irreducible non-negative matrix is primitive. For example, consider any permutation matrix.
\end{remark}

\begin{corollary} \label{CortoPrim}
 Let $A$ be a non-negative real $N \times N-$matrix, then the following statements hold.
\begin{enumerate}
    \item  The matrix $A$ is primitive if and only if  there exists some $0 < m \leq  N^2 - 2N +2$ such that $A^m$ has only positive entries iff $A^{N^2 - 2N +2}$ has only positive entries
(Horn-Johnson \cite{HoJo}, Cor 8.5.9; Huppert \cite{Huppert}, p.372).
\item If $A$ is a non-negative irreducible matrix, then we have:
If the trace of $A$ is positive,  then $A$ is primitive.
\item If $A$ is primitive, then 
\begin{equation}\label{Alim}
 (\frac{1}{\rho}A)^k  \xrightarrow[k \to \infty]{}  \frac{wv^{\rm tr}}{v^{\rm tr} w} > 0,
 \end{equation}
 where $w$ is a right Perron vector and $v$ is a left Perron vector for the Perron eigenvalue $\rho$.
 \item If $A$ is primitive, then  
 $\mathbb{R}^N = \mathbb{R}w \oplus v^{\perp},$ where $v$ is a left eigenvector of $A$ for the eigenvalue $\rho$ and $A v^{\perp} \subset v^{\perp}$.
\end{enumerate}
\end{corollary}
The last statement above follows from $(\ref{Alim})$ above
 by an application to $Az, z \in v^\perp$ . 
 Moreover, for a primitive  companion matrix $A$ the left eigenvector$v$ is a positive vector. Therefore the limit in (\ref{Alim}) 
is a positive projection.

If the hypothesis in the Perron-Frobenius Theorem is strengthened to assuming that $A$ is primitive, then we obtain (repeating some facts stated earlier):
\begin{theorem}[{\bf Perron-Frobenius Theorem for primitive matrices}]\label{primmatthm}
Assume that $A$ is a primitive non-negative $N \times N-$matrix. Then $A$ is irreducible and the Perron-Frobenius Theorem \ref{perfrothm} holds. In addition, the unique positive real Perron-Frobenius eigenvalue $\rho = \rho(A)$ for $A$ exceeds in magnitude all other 
(real or  complex) eigenvalues of $A$.
\end{theorem}

\begin{corollary}\label{limcor} If $A$ is a primitive non-negative matrix with spectral radius $\rho$, then the eigenspace for eigenvalue $\rho$ is spanned by some positive vector $w$, and for every $z\in\mathbb R^n$ the sequence $\left((\dfrac1\rho A)^k z\right)$ converges to a multiple of $w$. \end{corollary}
\begin{proof}
    This follows immediately from equation \eqref{Alim}, which implies
    \[
    \lim (\dfrac1\rho A)^k\,z= \frac{v^{\rm tr}z}{v^{\rm tr} w}\,w.
    \]
\end{proof}

\begin{remark}
    Assume that $A$ is primitive with $\rho=1$. Then $A$ is conjugate to an irreducible {column-}stochastic  matrix (see e.g. Huppert \cite{Huppert}, Section IV,4): Putting  $\mathbb{1} = (1,\dots ,1)^{tr}$ and choosing $v^{tr}$ as above, we write
    $v^{\rm tr}=\mathbb{1}^{tr}D$ with a diagonal matrix $D$ 
    and find
    $\mathbb{1}^{tr}D=v^{\rm tr}=v^{\rm tr}\,A =\mathbb{1}^{tr}\,D\,A,$
     whence  $\mathbb{1}^{tr}\,D\,A\,D^{-1}=\mathbb{1}^{tr}$ follows.
\end{remark}



\subsection{Constructing inflationary tilings from non-negative integer matrices}\label{subsec:nonneg}
In this subsection we consider exclusively  an irreducible non-negative matrix 
$U= \left(u_{ij}\right)\in \mathbb R^{N\times N}$, $N>1$, with integer entries, and  spectral radius $\rho>1$. We note that the latter property is automatic in the primitive case:
\begin{lemma} Let $U$ be a primitive non-negative $N \times N-$ integer matrix. Then the
    spectral radius of $U$ satisfies $\rho>1$.
\end{lemma}
\begin{proof}
    By Lemma \ref{lem:nilpo} the spectral radius is $\geq 1$.  Knowing that the spectral radius is also a real eigenvalue of $U$, we suppose now that this  eigenvalue is $1$. Then by the Perron-Frobenius Theorem there exists a positive eigenvector $v$ for this eigenvalue $1$, whence
    \[
\sum_{j = 1}^N u_{ij}v_j= v_i;\quad 1\leq i\leq N.
    \]
    Assume with no loss of generality that $v_1=\cdots= v_r$ are the smallest entries of $v$. Then, for $r<N$, the eigenvector condition shows that $u_{ij}=0$ for all $i\in \left\{1,\ldots,r\right\}$ and all $j\in \left\{r+1,\ldots,N\right\}$, and $U$ is reducible; a contradiction. In the remaining case, $r=N$, the matrix $U$ 
    is a permutation matrix; a contradiction to $U$ being primitive.
    \end{proof}

By the Perron-Frobenius theorem, the eigenspace for $\rho$ is one-dimensional and spanned by a vector 
\[
w=\begin{pmatrix}
    w_1\\ \vdots\\ w_N
\end{pmatrix}, \text{  with all } w_i>0,
\]
which may (but need not) be normalized by requiring $\sum w_i=1$.
We fix some ``basic length'' $L_0>0$ and give the name ``prototiles'' to the  intervals
$R_i:=[0,\,L_0 w_i]$.
With these definitions one recovers the relations in \eqref{partimat} of section 6.3; in particular the $R_i$ form a $\rho$-inflationary set of prototiles.\\
In view of Remark \ref{indicatorrem}, indicator sequences built with those prototiles (possibly augmented by $E$) define a tiling of the half-line, or compact intervals $[0,a]$, respectively.

Essentially the following step-by-step approach is a variant of an abstract procedure (``substitution of words") known from a more general context; see e.g.\ Barge and Diamond \cite{BarDia}, Section 2.
We will introduce some more notation, to distinguish clearly between various objects and operations. 

\begin{definition}
We consider the prototiles $R_i$, $1\leq i \leq N$.
\begin{itemize}
 \item    Given the infinite indicator sequence $\mathcal{X}$ of a finite tiling of some compact interval $[0,a]$, 
 we define its {\em cutoff} ${\rm cut}(X)$ to be the finite string which is obtained from discarding all the ``$E$'' entries.
 \item   Conversely, given a finite string $\mathcal S$, of length $k$, with entries in $\{R_1,\ldots,R_N\}$, we call its {\em completion} ${\rm comp}(\mathcal S)$ the sequence with the first $k$ entries identical to those of $\mathcal S$, and all remaining entries equal to $E$.
\item    For two strings $\mathcal S_1,\,\mathcal S_2$ of length $k_1,\,k_2$ respectively, their {\em join} $\mathcal S_1|\mathcal S_2$ is defined as the string of length $k_1+k_2$ with the first $k_1$ entries identical to those of $\mathcal S_1$, and the remaining ones identical to those of $\mathcal S_2$.
\end{itemize}
\end{definition}
\begin{definition}\label{def:ins} {\bf Inflation and substitution:}
Let $U$  be non-negative $N \times N$ matrix having spectral radius $\rho>1$, with associated
prototiles
$R_1, \ldots,R_N$
from the eigenvector $w$.
We consider the following procedure:
\begin{enumerate}[1.]
    \item {\bf Fixing a decomposition of inflated tiles: }For each $R_i$, $1\leq i\leq N$, define $v_i:=\sum_j u_{ij}$ and a finite string 
    \begin{equation}\label{insdecomp}
        {\rm ins}\,(R_i):=\left(S_{i,1},\ldots,S_{i,v_i}\right), \quad S_{i,k}\in\left\{R_1,\ldots,R_N\right\},
    \end{equation}
    such that
    \begin{equation}
        \#\left\{k:\,S_{i,k} = R_j\right\} = u_{ij},\quad 1\leq j\leq N.
    \end{equation}
    \item {\bf For each finite string $(R_{j_1},\ldots,R_{j_\ell})$ define}  ''ins"
    \begin{equation}
        {\rm ins}\,((R_{j_1},\ldots,R_{j_\ell})):=\left({\rm ins}\,(R_{j_1)}|\ldots|{\rm ins}\,(R_{j_\ell})\right).
    \end{equation}
    \item  The inflation and substitution procedure yields in the first step for the lengths $L(R_j)$ of the tiles 
     $R_1, \dots ,R_N$ the relation
     \begin{equation} \label{lengths}
     \rho L( R_k) = \sum_1^N u_{kj} L(R_j),
     \end{equation}
     where $\rho$ denotes the spectral radius of U.
    \item {\bf Defining ${\rm ins}\,({\mathcal X})$ for an infinite indicator sequence $\mathcal{X}$:}
    Represent $\mathcal X$ as a limit of completions of finite strings $\mathcal T_j$, and take the limit of the completions of the ${\rm ins}\,({\mathcal T_j})$. (One checks easily that this is well-defined.)
    \end{enumerate}
      We call ${\rm ins}$ an {\em inflation and substitution map}.
    \end{definition}
    
    \begin{remark}\label{instok}
        Obviously, for any $k>1$ one may consider ${\rm ins}^k$ as an inflation and substitution map in its own right, with matrix $U^k$ and substitution rules from ${\rm ins}^k(R_i)$.
    \end{remark}
    
We use these steps to construct $\rho$-inflationary tilings: Start with a finite string (``initial string'') of prototiles, and iterate the inflation-substitution procedure, to obtain a sequence of indicator sequences. If this sequence converges, then the limit is the indicator sequence of a $\rho$-inflationary tiling.
To discuss convergence, we note 
an obvious, but useful property:
\begin{lemma}\label{joinlem}
Let $\mathcal S$ and $\mathcal T$ be strings. Then
\[
{\rm ins}\,\left(\mathcal S|\mathcal T\right)={\rm ins}\,(\mathcal S)|{\rm ins}\,(\mathcal T).
\]
\end{lemma}
\begin{definition}\label{def:inscon}
     { \bf Recursive construction of tilings:} 
    Given the setting and the notation of Definition \ref{def:ins}, we construct a sequence of indicator sequences as follows.
    \begin{enumerate}[1.]
    \item {\em Choice of initial string:} Given a (finite) string $\mathcal S_1$, define
    \[
    \mathcal{X}_1:={\rm comp}(\mathcal S_1).
    \]
    \item {\em Inflation and substitution:} Given $\mathcal X_\ell$, set
    \[
    \mathcal S_\ell:={\rm cut} (\mathcal{X}_\ell),
    \]
    and define
    \[
    \mathcal{X}_{\ell+1}:={\rm comp}({\rm ins}\,(\mathcal S_\ell)).
    \]
\end{enumerate}
  We will also call this construction the {\em inflation-substitution iteration}.
   \end{definition}

\begin{theorem}\label{conthm}
\begin{enumerate}
    \item If the sequence $(\mathcal{X}_\ell)_{\ell\in\mathbb N}$ converges, then one obtains a tiling of the half-line with prototiles $R_1, ..., R_N$, and 
$\lim \mathcal X_\ell$ as indicator sequence.
\item Moreover, if there exists some integer $d>1$ such that the subsequence $(\mathcal{X}_{1+ dm})_{m\in\mathbb N_0}$ converges, then one obtains a tiling of the half-line with prototiles $R_1, ..., R_N$, and 
$\lim \mathcal X_{1+dm}$ as indicator sequence.
\end{enumerate}
\end{theorem}
We point out that by Lemma \ref{joinlem}, only any initial part of the initial string matters. More precisely: Given initial strings $\mathcal S_1$ and $\mathcal S_1^*=\mathcal S_1|\mathcal T$ for some $\mathcal T$, one sequence converges if and only if the other does, and the limits are equal.
We note  a consequence of \eqref{lengths}.
\begin{lemma}\label{lem:count}
    Let $\mathcal S$ be a finite string, and for each prototile $R_i$ denote by $m_i$ the number of occurrences of $R_i$ in $\mathcal S$. Then the respective numbers of occurrences in the string $\widetilde {\mathcal S}$ obtained from one inflation-substitution step are obtained from the transpose of the partition matrix, via
    \begin{equation}
    U^{\rm tr}\cdot     \begin{pmatrix} m_1\\ \vdots\\ m_N
        \end{pmatrix}.
    \end{equation}
\end{lemma}
\begin{proof}
    According to \eqref{lengths}, in the inflated $\rho R_k$ one has $U_{kj}$ appearances of $R_j$. So the total number of appearances of $R_j$ in the string $\widetilde{\mathcal S}$ equals $\sum_k m_ku_{kj}$.
\end{proof}
\subsection{About convergence} \label{convergence}
We start with a basic convergence result.

\begin{proposition}\label{stableprop}
Given the situation of Definition \ref{def:inscon}, assume that the sequence $(\mathcal X_\ell)_{\ell\geq 1}$ has initial string $\mathcal S$ such that ${\rm ins}\,(\mathcal S)=\mathcal S|\mathcal T$ for some string $\mathcal T$. Then
we obtain
\begin{enumerate}
\item With $\mathcal S_{1}:=\mathcal S$, one has $\mathcal S_{\ell+1}=\mathcal S_\ell|\mathcal T_\ell$ for all $\ell$, 
\item The sequence $(\mathcal X_\ell)_{\ell\geq 1}$ converges.
\item The limit 
$$\lim_{\ell}\mathcal{X}_{\ell} = \mathcal{X}$$ is stable under inflation and substitution, i.e.
$${\rm ins}(\mathcal{X}) = \mathcal{X}.$$
\end{enumerate}
Thus one obtains an inflationary tiling with multiplier $\rho$.
\end{proposition}
\begin{proof}
It is sufficient to show 
$\mathcal S_{\ell+1}=\mathcal S_\ell|\mathcal T_\ell$ for all $\ell$. For $\ell=1$ this is clear, and we proceed by induction. Assuming that this holds for $\ell$, we have with Lemma \ref{joinlem}:
\[
\mathcal S_{\ell+2}={\rm ins}\,(\mathcal S_\ell|\,(\mathcal T_\ell)={\rm ins}\,(\mathcal S_\ell)|{\rm ins}\,(\mathcal T_\ell)=\mathcal S_{\ell+1}|\mathcal T_{\ell+1}
\]
with $\mathcal T_{\ell+1}:={\rm ins}\,(\mathcal T_\ell)$.
\end{proof}
\begin{remark}
    In the setting of the Proposition, assume w.l.o.g.~that $\mathcal S=(R_1,\ldots)$. Then, letting $\widetilde{\mathcal S}=(R_1)$, one has ${\rm ins}\,(\widetilde{\mathcal S})=\widetilde{\mathcal S}|\widetilde {\mathcal T}$ for some $\widetilde {\mathcal T}$, thus the sequence $(\widetilde{\mathcal X}_\ell)$ constructed from the initial string $(R_1)$, converges, with the same limit as $(\mathcal X_\ell)$. To see the latter, notice that from $\mathcal S=\mathcal{\widetilde S}|\mathcal V$, for some $\mathcal V$ one obtains
    \[
    {\rm ins}^m\,(\mathcal S)={\rm ins}^m\,(\mathcal{\widetilde S})|{\rm ins}^m\,(\mathcal V)
    \]
    for all $m$, and the leftmost position of the entries in the second string tends to infinity.\\
    Thus, to discuss convergence or divergence, given a matrix $U$ and
    substitution rules, it suffices to consider single-entry initial strings.
\end{remark}

Moreover, with Remark \ref{instok} we see:
\begin{corollary}
    Given the situation of Definition \ref{def:inscon}, and an integer $k>1$, assume that the sequence $(\mathcal X_\ell)_{\ell\geq 1}$ has initial string $\mathcal S$ such that ${\rm ins}^k\,(\mathcal S)=\mathcal S|\mathcal T$ for some string $\mathcal T$. Then
the limit 
$$\lim_{m\to\infty}\mathcal{X}_{1+km} = \widehat{\mathcal{X}}$$ satisfies
$${\rm ins}^k(\widehat{\mathcal{X}}) = \widehat{\mathcal{X}}.$$
Thus one obtains an inflationary tiling with multiplier $\rho^k$.
\end{corollary}

We note some special consequences, expressed in the language of tilings: 
\begin{corollary}\label{convcor} Let $U$ be an irreducible non-negative integer matrix having spectral radius $>1$, and  $1\leq j\leq N$.
\begin{enumerate}
    \item If the decomposition of $\rho R_j$, according to \eqref{insdecomp} in Definition \ref{def:ins}, starts leftmost with $R_j$ (thus, necessarily $u_{jj}>0$ in the matrix $U$, and $U$ is primitive by Theorem \ref{primmat}), then the initial string with single entry $R_j$ satisfies the hypothesis of Proposition \ref{stableprop}, and yields a convergent sequence. The limit is inflationary with multiplier $\rho$. 
    \item If, after $k>1$ iterations, the decomposition of $\rho^k R_j$, according to \eqref{insdecomp} in Definition \ref{def:ins}, starts leftmost with $R_j$ (thus, necessarily the matrix $U^k$ has positive $(j,j)$-entry and is primitive), then taking the initial string with single entry $R_j$, the subsequence with indices $1+km$, $m\geq 0$ converges and yields an inflationary tiling with multiplier $\rho^k$.
\end{enumerate}
\end{corollary}
\begin{proof}
    The first statement is clear. For the second note that $k$-fold iteration amounts to an inflation-substitution construction with matrix $U^k$.
\end{proof}
\begin{remark}
    Substantial parts of the corollary hold true even when $U$ is a non-negative integer matrix with spectral radius $\rho>1$, and a positive eigenvector $w$ for $\rho$, mimicking the construction in Definitions \ref{def:ins} and \ref{def:inscon}. A crucial consequence of primitivity is that in part (2) the existence of $k$ is guaranteed for each $j$. One could also proceed and obtain a different proof of Theorem \ref{classifyprimitive}.
\end{remark}

We consider some examples and non-examples for the inflation-substitution iteration.
\begin{example}\label{instileex}
    \begin{enumerate}[(a)]
        \item {\em Distinguished Resnikoff silver numbers: }Let $\rho$ satisfy \eqref{dist-silverpoly}. Then, in particular, 
        the  degree $N$ of $\rho$ is $\geq 2$. From the partition matrix $U_N$  we obtain prototiles $R_i$ of respective length $\rho^{1-i}$, $1\leq i\leq N$, with the sum of lengths equal to $\rho$. The decomposition of inflated tiles is uniquely determined for all $i>1$, since
        \[
        \rho R_i=R_{i-1}, \quad 2\leq i\leq N,
        \]
        while, with any permutation $\pi$ of $\{1,\ldots,N\}$,
        \[
        \rho R_1=R_{\pi(1)}|\cdots|R_{\pi(N)}.
        \] 
        According to Proposition \ref{stableprop}, whenever $\pi(1)=1$ we obtain a convergent sequence of tilings, with a $\rho$-inflationary limit. In case $\pi(1)\not=1$, and $k$ minimal with $\pi^k(1)=1$, one obtains $\rho^k$-inflationary tilings from subsequences.
        Note that for $N=2$ (the golden number) only one convergent sequence exists. 
        \item  The  {\em super-golden number} $\psi,$  with $\psi^3 - \psi^2 - 1 =0$ yields prototiles $R_i$ with lengths
        $\psi^{1-i}$, respectively, $1\leq i\leq 3$. In the inflation-substitution process the decomposition of inflated tiles yields two possible cases:
        Either  $\psi R_1=R_1|R_3,\,\psi R_2=R_1,\,\psi R_3=R_2$ or $\psi R_1=R_3|R_1,\psi R_2=R_1,\psi R_3=R_2$.
        The first of these yields a convergent sequence and produces a $\psi$-inflationary tiling; for the second, passing to a suitable subsequence will produce a $\psi^3$-inflationary tiling.
         \item The  {\em  plastic number} $\theta$, with $\theta^3-\theta-1=0$ yields prototiles $R_i$ with lengths $\theta^{1-i}$, respectively, $1\leq i\leq 3$. In the inflation-substitution process the decomposition of inflated tiles yields two possible cases:
        Either  $\theta R_1=R_2|R_3,\,\theta R_2=R_1,\,\theta R_3=R_2$ or $\theta R_1=R_3|R_2,\,\theta R_2=R_1,\,\theta R_3=R_2$.
      As stated in full generality  in  item (1) of Proposition \ref{divrem} we observe that none of these yields a convergent sequence producing a $\theta$-inflationary tiling, since there is no $j$ such that $\theta R_j=R_j|\cdots$.
      By passing to a suitable subsequences, however, the first substitution noted above will produce a $\theta^2-$inflationary tiling, and the second will produce a $\theta^3$-inflationary tiling.
       (We observe that $\theta^2$ has minimal polynomial $X^3-2X^2+X-1$; so $\theta^2$ is not a root of an irreducible silver polynomial.
       Likewise, one sees that $\theta^3=\theta +1$ is not a root of an irreducible silver polynomial, because its minimal polynomial is $X^3-3X^2-2X+1$).
  \end{enumerate}
\end{example}
\subsection{Matrix transformations versus tilings}


\subsubsection{A review of companion matrices}\label{subsec:compmat}
Consider a polynomial 
\begin{equation}
Q(X)=X^N-\sum_{j=1}^Nc_jX^{N-j}\in K[X],
\end{equation}
 over a field $K$. For reasons of convenience we will assume that $K$ is of characteristic zero, hence infinite.
We will call the matrix 
\begin{equation} \label{CDW}
C_{DW}=\begin{pmatrix} c_1&c_2&\cdots&\cdots&c_N\\ 
            1 & 0& & &0\\
        0&\ddots&  & & 0\\
        \vdots & & & &\vdots\\
    0&\cdots & & 1&0
\end{pmatrix},
\end{equation}
(see \cite{ChenLouck} for the notation), the {\em Frobenius companion matrix} of the polynomial $Q$,
following the terminology in Resnikoff \cite{Res}. We  abbreviate the first row of $C_{DW}$ by  $c^{tr}$. 

\begin{remark}
In Section \ref{partitionmatrixofsilver} we have discussed partition matrices of  this type and the inflationary tilings associated to them.
\end{remark}

\begin{remark} \label{companiontypes}
There is no unique notion of a ”companion matrix” of a polynomial or of  a  matrix, meaning in the latter case
the companion matrix of its characteristic polynomial. 
The following four types of a matrix appear in the literature:
\begin{enumerate}[(i)]
\item $C_{DW}$ as stated just above.
 \item 
 \begin{equation} \label{CDWtranpose}
C_{DW}^{\rm tr}=\begin{pmatrix} c_1&1&0&\cdots&0\\ 
            c_2 & 0& 1& &0\\
        \vdots& \vdots& \ddots & \ddots& 0\\
        c_{N-1} &0 &\cdots &0 &1\\
    c_N&0 & \cdots&\cdots &0
\end{pmatrix}
\end{equation}
with the first column $c$. See also Section \ref{partitionmatrixofsilver} for partition matrices of  this type and the inflationary tilings associated to them.
\item \begin{equation} \label{CPmat}
C_{P}=\begin{pmatrix} 0& 1&0\dots& &0\\ 
            0& 0&1\dots & &0\\
        0&\ddots&  & & 0\\
        \vdots & & & &\vdots\\
        0&\cdots & & 0&1 \\
           c_N&\cdots & & c_2&c_1
\end{pmatrix} = JC_{DW} J,
\end{equation}
where $J$ is the ”skew-diagonal” permutation matrix with the only 
non-vanishing entries
$J_{nm} = 1,$ iff $n + m = N + 1$. 
This type of companion matrix is used in Prokip \cite{Prokip}; the last row of $C_P$ equals $(c_N, \dots , c_1)=c^{\rm tr}J$. 
  \item \begin{equation} \label{CPtrans}
   C_{P}^{\rm tr}=\begin{pmatrix}
        0&0&\cdots&\cdots&0&c_N\\
        1&0&0&\cdots& 0 &c_{N-1}\\
        0&\ddots &\ddots & &\vdots & \vdots\\
        &&&&&\\
        \vdots & & & \ddots& 0& c_2\\
        0 & \cdots&\cdots & 0& 1&c_1
    \end{pmatrix} = JC_{DW}^{\rm tr}J;
\end{equation}
 this is also known as $C_M $, see e.g. Meyer \cite{Meyer}; its last column is the vector $J {c}$.  
  \end{enumerate}
All these matrices have the characteristic polynomial $Q$, and they are invertible if and only if $c_N\not=0$. Note that they are obtained from each other by conjugation with a permutation matrix, possibly combined with transposition.\end{remark}
The following facts about eigenvectors follow by straightforward computations (compare Remark \ref{lenghthoftiles}):
\begin{lemma}\label{lem:evec}
Let $\sigma$ be a nonzero root of $Q$. Then 
the column vector
 \begin{equation} \label{equeigenv}
    v:=  ( \frac{1}{\sigma^{n-1}}, \frac{1}{\sigma^{n-2}}, \dots, \frac{1}{\sigma}, 1)^{tr}
 \end{equation}
 is an eigenvector for $C_P$ for the eigenvalue $\sigma$.\\
Moreover, an eigenvector for the eigenvalue $\sigma$ of $C_P^{tr}$ is of the form
\begin{equation}\label{evectrafo}
 \begin{pmatrix} 
 c_N/\sigma \\
(\sum_{j=0}^{m-1} c_{m-j}\sigma^j)/\sigma^m , {2\leq m\leq N -1} \\ 
(\sum_{j=0}^{N-1} c_{N-j}\sigma^j)/\sigma^N \\1 
\end{pmatrix} =
\begin{pmatrix}
c_N/\sigma \\ (c_N+c_{N-1}\sigma)/\sigma^2 \\ (c_N+c_{N-1}\sigma+c_{N-2}\sigma^2/\sigma^3\\ \vdots \\ 1
\end{pmatrix} = T\cdot v
\end{equation}
with
\begin{equation}\label{Tmatrix}
T:=\begin{pmatrix}
    0&0&0& \dots &0&c_N&0\\
    0&0&0&\dots&c_N& c_{N-1}&0\\
     0&0&\dots&c_N&c_{N-1}&c_{N-2}&0&\\
    \vdots &&&&& \vdots\\
    c_N&c_{N-1}& \dots & \dots & \dots & c_2&0\\
    0&0&\dots & \dots &\dots &\dots &1\\
\end{pmatrix}.
\end{equation}
\end{lemma}
 \begin{proposition}
Let $C_P$ be invertible; equivalently $c_N\not=0$. Then
\begin{equation} \label{conjcomp}
 C_P^{\rm tr} = TC_PT^{-1}.
 \end{equation}
with the invertible matrix $T$ from \eqref{Tmatrix}.
\end{proposition}
\begin{proof}
The essential arguments in the proof go back to Prokip \cite{Prokip}.\\
We first note that $\det T=\pm c_N^{N-1}$  shows invertibility of $T$.
\begin{enumerate}[(i)]
\item Preliminary consideration: The polynomial $Q$ has only simple roots in its splitting field if and only if its discriminant $\Delta$ is nonzero; see e.g. Lang \cite{Lang}, IV \S8, Corollary 8.4 and Proposition 8.5. The discriminant is a polynomial in $(c_1,\ldots, c_N)\in K^N$, and nonzero at $c_1=\cdots=c_{N-1}=0$, $c_N=1$ (when the roots are the $N$ distinct  $N^{\rm th}$ roots of unity). So the discriminant is nonzero on a Zariski-open and dense subset of $K^N$.
\item Suppose that the discriminant is nonzero, and let $\sigma_1,\ldots,\sigma_N$ be the pairwise distinct (and nonzero) roots of $Q$ in some extension field of $K$. Then by Lemma \ref{lem:evec} we have an eigenvector $v_j$ for the eigenvalue $\sigma_j$ of $C_P$, and $Tv_j$ is an eigenvector for $C_P^{\rm tr}$. With Lemma \ref{lem:evec} we have the identities 
\[
TC_P\,v_j=\sigma_j\,Tv_j=C_P^{\rm tr}Tv_j, \quad 1\leq j\leq N,
\]
and for the invertible matrix $V:=(v_1,\ldots,v_N)$ this implies $TC_PV=C_P^{\rm tr}TV$, which in turn shows that
\begin{equation}\label{transpoconj}
TC_P=C_P^{\rm tr} T.
\end{equation}
\item For the general case, note that \eqref{transpoconj} is equivalent to a set of polynomial equations in the $c_j$ (consider each matrix entry of the left and right hand side), which holds on a Zariski dense subset of $K^N$, therefore on all of $K^N$. Since $c_N\not=0$, $T$ is invertible.
\end{enumerate}
\end{proof}
\begin{remark}
    For the cases $\mathbb K=\mathbb R$ or $\mathbb K=\mathbb C$, one may avoid the Zariski topology and modify the above proof by considering non-invertible matrices as limits of invertible ones. One just has to use that the vanishing set of any nonzero polynomial is open and dense in the norm topology.
\end{remark}
\begin{remark}  We recall the general  result (see e.g. \cite{Kaplansky}, Theorem 66), that
any matrix $Y \in Mat(N,K)$ is conjugate to its transpose $Y^{tr}$ by an  invertible symmetric matrix in $GL(N,K).$  But for our purpose the special form of the conjugating matrix for $C_P$ and its transpose are relevant.
In case $K=\mathbb Q$, and integers $c_j\geq 0$, our argument shows that the matrix $T$ has non-negative integer entries, and this property will be crucial in the following subsections. \\    
\end{remark}
\subsubsection{Conjugations mapping a non-negative matrix to a non-negative matrix}
Our primary focus is on partition matrices and the tilings related to them. Generally one may ask how to produce ``new'' tilings from given ones.  A first step is to ask how to produce ``new'' partition matrices from given ones. In the present subsection we collect some facts and observations about non-negative matrices. The first lemma is obvious.

\begin{lemma}\label{trafolem}
Let $U$ be a non-negative matrix with integer entries. Then:
\begin{itemize}
\item For every permutation matrix $W$, the matrix $WUW^{-1}$ is a non-negative integer matrix, which has the same characteristic and minimal polynomials as $U$. If $U$ is irreducible, resp.\ primitive, then the same holds for $WUW^{-1}$.
\item The transpose $U^{\rm tr}$ is a non-negative integer matrix, which has the same characteristic and minimal polynomials as $U$. If $U$ is irreducible, resp.\ primitive, then the same holds for $U^{\rm tr}$.
\end{itemize}
\end{lemma}
There is an obvious question how the tilings produced from these related matrices correspond to tilings  produced from $U$. In the first case, conjugation with a permutation matrix, the answer is straightforward: For tilings, the transformation amounts to a renumbering of the prototiles (given by the permutation underlying $W$). A more general answer will be given below.

As to conjugacy of non-negative matrices, one may first be interested in conjugation by matrices that map the positive orthant $\mathbb R_+^N$ onto itself, i.e. non-negative matrices with non-negative inverse. Clearly, conjugation by any such matrix will send non-negative matrices to non-negative matrices. Moreover, with tilings in mind, one is primarily interested in conjugating matrices that have integer entries. But the set of these matrices is rather small.
\begin{lemma}
Let $W$ be a real invertible $N\times N$ matrix that maps the positive orthant $\mathbb R_+^N$ onto itself. Then
\begin{enumerate}
\item $W$ is an automorphism of the positive orthant and $W^{-1}$ is  non-negative.
Moreover, $W=P\cdot D$, with $P$ a permutation matrix and $D$ a diagonal matrix with positive entries.
\item If both $W$ and $W^{-1}$ have integer entries, then they are permutation matrices.
\end{enumerate}
\end{lemma}
\begin{proof}
By assumption, $W$ is a linear map mapping the first orthant onto itself bijectively. Hence $W^{-1}$ is also a non-negative matrix.
The first statement follows from evaluating the equation $W W^{-1} = I.$
The second then is obvious.
\end{proof}
We state a generalization of \ref{trafolem} that requires less restrictive hypotheses.
\begin{proposition} \label{weaktrafolem}
    Let $A$ be an irreducible non-negative $N\times N$ matrix with integer entries and spectral radius $\rho$. Moreover let $W$ be a non-negative matrix with integer entries which is invertible as a rational matrix and such that 
    \begin{equation}\label{conjugatemat}
        B:=WAW^{-1}
    \end{equation} is non-negative with integer entries, and irreducible. Let $y=(y_1,\ldots,y_n)^{\rm tr}$ be a positive (right) eigenvector of $A$, defining the prototile lengths for some inflation-substitution iteration. Then the entries of $Wy$ form a set of prototile lengths for an inflation-substitution iteration with partition matrix $B$. Each of the latter prototiles is an essentially disjoint union of some of the former.
\end{proposition}
The proof is obvious, since $Wy$ is a positive eigenvector of $B$. 
\begin{remark}
\begin{enumerate}
\item By Theorem \ref{perfrothm}  the vector $y$ is uniquely determined up to a positive factor.
\item This result may be seen to represent a natural relation between the prototiles and, by extension, between the tilings associated to $A$ and $B$ respectively.
\item For any two matrices $W$ and $\hat{W}$
satisfying all assumptions of the proposition, we see that $Z = W^{-1} \hat{W}$ is a rational matrix commuting with $A$. If characteristic polynomial and minimal polynomial of $A$ coincide, then $Z$ is a linear combination of powers of $A$, as known from linear algebra.
\end{enumerate}
\end{remark} 
Proposition \ref{weaktrafolem} assumes the existence of a matrix $W$ such that \eqref{conjugatemat} holds, but says nothing about possible constructions. We turn to this topic now, with some additional requirements. We will freely use the companion matrix types from Remark \ref{companiontypes}, but we will restrict attention to non-negative real matrices.
\begin{theorem} \label{conjAtoC_P}
Let $A$ be an irreducible, non-negative, integer $N \times N-$matrix that is invertible over $\mathbb Q$.
Assume that its characteristic polynomial has the form
\[
P=X^N-\sum b_jX^{N-j}
\]
with integers $b_j\geq 0$, and is equal to its minimal polynomial.
Then there  exist  non-negative  integer matrices $W$ and $\hat{W}$ which  are invertible considered as  rational matrices 
\begin{equation} \label{fromAtoC}
    C_P = W A W^{-1}  \hspace{2mm} \mbox{and} \hspace{2mm} 
   C_P^{tr} = \widehat{W}^{-1} A \widehat{W},
\end{equation}
where $C_P$ is given in \eqref{CPmat} with $b_j$ replacing $c_j$, $1\leq j\leq N$, and also for its transpose in \eqref{CPtrans}; see Remark \ref{companiontypes}.
Furthermore, $C_P$ is an irreducible, non-negative, integer matrix with characteristic polynomial
$P_A$. 
\end{theorem}
\begin{proof}
It is well known (see e.g \cite{Prokip}, Theorem 1), that there exists an invertible rational matrix 
$W$ such that $C_P = W A W^{-1}$, viz.\ $W = W(u)$ defined by $W^{tr} = (u, A^{tr} u, \dots, (A^{tr})^{N-1} u),$ with 
$u \in \mathbb{Q}^N$ such that $\det(W(u))  \neq 0.$ Since the determinant condition can also be satisfied with some $u>0$, we may assume that $u$ has non-negative integer entries, and we  obtain that $W$ is a non-negative integer matrix, and 
invertible as a rational matrix. Likewise, there exists a vector $v$ with non-negative integer entries such that the non-negative integer matrix  $\widehat W:=(v,Av,\ldots,A^{N-1}v)$ is invertible over $\mathbb Q$. By construction this matrix satisfies $A\widehat W=\widehat W C_P^{\rm tr}$.\\
The irreducibility of $C_P$ follows from the invertibility of $A$, which implies the invertibility of $C_P$, equivalently $b_N\not=0$, and the criterion in \cite{HoJo}, Theorem 6.2.24, part (c).
\end{proof}


\subsubsection{Application to partition matrices and related tilings}
\begin{example}
In previous subsections we discussed inflation-substitution iterations from companion matrices of the special type
\[
C_{DW}=\begin{pmatrix} b_1&b_2&\cdots&\cdots&b_N\\ 
            1 & 0& & &0\\
        0&\ddots&  & & 0\\
        \vdots & & & &\vdots\\
    0&\cdots & & 1&0
\end{pmatrix}
\]
with non-negative integer entries $b_1,\ldots,b_N$.
But one also obtains inflation-substitution constructions 
for the other types of companion matrices  noted in Remark \ref{companiontypes}. We consider explicitly the constructions  for 
    \begin{equation*}
C_{DW}^{\rm tr}=\begin{pmatrix} b_1&1&0&\cdots&0\\ 
            b_2 & 0& 1& &0\\
        \vdots& \vdots& \ddots & \ddots& 0\\
        b_{N-1} &0 &\cdots &0 &1\\
    1&0 & \cdots&\cdots &0
\end{pmatrix}.
\end{equation*}
Straightforward application of the inflation-substitution procedure yields one iteration by setting, for $i<N$,
\[
\rho R_i= \left\{\begin{array}{ccccc}
    R_1|R_{i+1}& & \text{if}& & b_i=1;\\
    R_{i+1} & & \text{if} & & b_i=0.
\end{array}
\right.
\]
and finally $\rho R_N=R_1$. 
Every other iteration rule is obtained by modification of this one by stipulating $\rho R_i=R_{i+1}|R_1$ for certain indices $i<N$ with $b_i=1$. 
\end{example}
The following is the main result of the present subsection.
\begin{theorem} \label{conjAtoB}
    Let $A$ and $B$ be irreducible non-negative integer matrices such that their common characteristic polynomial equals their minimal polynomial, and has the form
$P(X)=X^N-\sum b_jX^{N-j}$ with non-negative integers $b_j$, and $b_N\not=0$.
Then there exist non-negative integer matrices $W$, $\widehat{W}$ and $T$ which are invertible as rational matrices such that
    \begin{equation}
        B (\widehat{W} T W) =  (\widehat{W} T W) A.
    \end{equation}
    An analogous equation holds with $A$ and $B$ interchanged.
\end{theorem}
\begin{proof}
Choose $W$ and $\widehat{W}$, according to Theorem \ref{conjAtoC_P}, with $WA=C_PW$ and $B\widehat W=\widehat WC_P^{\rm tr}$, and $T$ as in $(\ref{conjcomp}$, satisfying $TC_P=C_P^{tr} T$.
Then we infer
\[
\widehat WTWA=\widehat WTC_PW=\widehat WC_P^{\rm tr}TW=B\widehat WTW,
\]and the claim follows.
\end{proof}
\begin{corollary}
Let $A$ and $B$ be irreducible non-negative integer matrices such that their common characteristic polynomial equals their minimal polynomial, and is of the form $X^N-\sum b_jX^{N-j}$ with non-negative integers $b_j$,  $b_N\not=0$. Then, given any set of prototiles for $B$, their lengths are (up to scaling) non-negative integer linear combinations of prototile lengths for any tiling obtained from $A$.
\end{corollary}
\begin{proof}
Due to the correspondence of prototile lengths and entries of positive eigenvectors, we need to show the following: Given a positive right eigenvector $w$ of $A$, $Aw=\rho w$ with the spectral radius $\rho$, some non-negative integer linear combinations of the entries of $w$ are entries of a positive  eigenvector of $B$. In other words, we need to show that there is a non-negative integer matrix $M$ such that $Mw$ is an eigenvector of $B$. But by Theorem \ref{conjAtoB} there exists a non-negative integer matrix $M$ such that $BM=MA$, which implies $BMw=MAw=\rho\,Mw$.
\end{proof}
\begin{remark}
Keeping the assumptions of the corollary, one obtains prototiles for $B$ from any set of prototiles for $A$ by joining suitable numbers of these. Note that one may reverse the roles of $A$ and $B$, so the statement goes both ways. From this perspective one gains a partial understanding of tilings constructed from matrices $A$ that satisfy the hypotheses of Theorem \ref{conjAtoB}. Prototiles constructed from  a conjugate companion matrix are well understood (see Lemma \ref{lem:evec}), and prototiles for $A$ are obtained by joining these in a suitable way.
\end{remark}
\begin{remark}
    However, the hypothesis that not only $A$, but also its companion matrix is non-negative,  imposes a restrictive additional requirement. (One verifies, however, that the set of non-negative matrices with this property contains a non-empty open subset of $\mathbb R^{n\times n}$.) If this is not satisfied, then $A$ still has an integer companion matrix. Theorem  \ref{conjAtoB} holds with integer matrices $W$, $\widehat W$ and $T$, and the former two are non-negative, but not $T$. Thus there is no immediate interpretation with regard to tilings, and more general interpretations (should they exist) would lead to different questions.
\end{remark}

We close this section with one example where the conjugating matrix is obtained by inspection, not via the argument in the proof of Theorem \ref{conjAtoC_P}.
\begin{example}
Let $c>1$ be an integer and
\[
A:=\begin{pmatrix}
    1&c&1\\ 1&0&0\\ 0&1&0
\end{pmatrix}, \quad S:=\begin{pmatrix}
    1&0&0\\ 0&1&0\\ 1&1&1
\end{pmatrix}\quad\text{with  }S^{-1}=\begin{pmatrix}
    1&0&0\\ 0&1&0\\ -1&-1&1
\end{pmatrix}.
\]
Here one computes
\[
B:=SAS^{-1}=\begin{pmatrix}
    0&c-1&1\\ 1&0&0\\ 1&c&1
\end{pmatrix}.
\]
With a positive eigenvector $w$ of $A$ one obtains the positive eigenvector $Sw$ of $B$, and thus from prototiles corresponding to $A$ one obtains prototiles for $B$.
\end{example}


\subsection{Non-periodicity}

The first proof of non-periodicity for inflationary tilings seems to go back to Penrose \cite{Pen}, \cite{PenPen}; see Resnikoff \cite{Res}. The following proof for tilings of the half-line uses properties of non-negative matrices.
We recall  Definition \ref{def:inscon} and Theorem \ref{conthm}.
\begin{theorem} [Penrose's proof by irrationality]\label{penrose} A tiling $(S_i)_{i \in \mathbb N}$ constructed from 
a primitive non-negative integer matrix $U$ with irrational spectral radius $\rho$  is not periodic.
\end{theorem}
\begin{proof} Assume that the tiling is periodic. Then there exists a subinterval $J=[0,x_0]$, corresponding to some initial string, such that all integer multiples of $x_0$ are endpoints of tiles, and that $x_0$ is the smallest such number. By appropriate scaling of the basic length $L_0$ we may assume  $x_0=1$.
\begin{enumerate}[1.]
\item Let $R_1,\ldots,R_N$ be the prototiles of the tiling. For each $j \in \{1,\ldots,N\}$ denote by $d_j$ the number of occurrences of $R_j$ in $J = [0,1]$, and set $f_j=d_j/(\sum d_k)\in \mathbb Q$ . Note, summation in the denominator will always be from $1$ to $N$. We observe $1 = \sum_{j=1}^N d_j L(R_j).$
Now consider for each $j \in \{1,\ldots, N \}$  the step-function
 \begin{equation}
 n_j : [0, \infty[ \rightarrow \mathbb{N} \cup \{ 0 \}, 
 \end{equation}
 with $n_j(b)$ defined as the number of times the closed prototile $R_j$ is fully contained in the interval $[0, b]$, with $b$ any non-negative  real number.
Since $1 = x_0$, for integers $m$ we obtain
\[
n_j(m)=m\cdot d_j, \quad 1\leq j\leq N,\quad \dfrac{n_j(m)}{\sum n_k(m)}=f_j,
\]
by periodicity. 
 Writing  any non-negative real number  $b$ in the form $b = m + \hat{b}$ with an integer $m$ and 
 $0\leq \hat{b} < 1$, we have, since with $x_0 = 1$ no positive integer is contained inside a prototile,
 \[
 n_j(b)=n_j(m)+n_j(\hat b)=m\cdot d_j+n_j(\hat b); \quad  0\leq n_j(\hat b) \leq d_j,
 \]
and we obtain the fraction 
  \begin{equation}
 \dfrac{n_j (b)}{\sum n_kb)} = \dfrac{md_j + n_j (\hat{b})}{m\sum d_k + \sum n_k(\hat{b})} =\dfrac{d_j + n_j (\hat{b})/m}{\sum d_k + \sum n_k(\hat{b})/m}.
 \end{equation}
Since each $n_k(\hat{b})$ is bounded by $d_k$, the limit as 
 $b \rightarrow \infty$ exists, and we obtain:
\begin{equation}\label{ratlimneu}
  \lim_{b\to\infty}\dfrac{n_j (b)}{\sum n_k(b)}=f_j\in\mathbb Q, \quad 1\leq j\leq N.
\end{equation}
    \item Now we use Lemma \ref{lem:count}. 
    Since $R_k$ is contained in the starting interval $J = [0,1]$ for $d_k$ times, the occurrences of the prototiles in the inflated first interval $\rho J$ are given by 
    \[
    U^{\rm tr}\cdot \begin{pmatrix}
        d_1\\ \vdots \\d_N
    \end{pmatrix},
    \]
    and $k$-fold iteration yields 
    \begin{equation} \label{iterationfrequencesinU}
       d^{(k)} =\begin{pmatrix}
           d_1^{(k)}\\ \vdots \\  d_N^{(k)}
       \end{pmatrix} = (U^{\rm tr})^k d^{(0)} ,
      \end{equation}
 with $d^{(0)} = (d_1, \dots d_N)^{tr}$.

In terms of $n_j$ we obtain for    $b = \rho^k$ the relation 
 $n_j(\rho^k) = d_j^{(k)}$.
 Inserting this into equation $(\ref{ratlimneu})$ we thus find
      \begin{equation}
     \frac{d^{(k)}_i}{\sum_j d_j^{(k)}}        \xrightarrow[k \to \infty]{}    \frac{d_i}{\sum_j d_j} =f_i, 
\end{equation}
    or, in vector form,
       \begin{equation} \label{translatesection1}
   \frac{d^{(k)}}{\sum_j d_j^{(k)}}        \xrightarrow[k \to \infty]{}    \frac{d^{(0)}}{\sum_j d_j} =:f. 
 \end{equation} 
 \item We apply Corollary \ref{limcor} to the matrix $A:=\dfrac1\rho U^{\rm tr}$, with positive eigenvector $v$ for $\rho$. Since $\rho$ is irrational, no nonzero multiple of $v$ has rational entries. In particular we have that 
 \[
 A^k\,\begin{pmatrix}
        d_1\\ \vdots \\d_N
    \end{pmatrix}
    \to \mu\,v
 \]
 for some $\mu\not=0$.
   \item On the other hand, we have
   \[
   A^k\,\begin{pmatrix}
        d_1\\ \vdots \\d_N
    \end{pmatrix}  =\dfrac{\sum_j d_j^{(k)} }{\rho^k}\cdot\left({U^{\rm tr}}\right)^k\left(\dfrac{1}{\sum_j d_j^{(k)}} d^{(k)}\right)
   \]
   Here the second term on the right-hand side converges to $f$, and from $f\not=0$ and  $\mu v\not=0$ one infers that the factor $\sum_j d_j^{(k)}/\rho^k$ converges to some nonzero real number.\\
   Altogether we see that $f\in \mathbb R v$; a contradiction, since the entries of $f$ are rational.
This proves the claim.
 \end{enumerate}
\end{proof}

\end{document}